\documentclass[9pt,a4paper]{article}
\usepackage{amsmath}
\usepackage{amsfonts}
\usepackage{amsthm}
\usepackage{amssymb}
\usepackage{booktabs,amsmath}

\def\XXint#1#2#3{{\setbox0=\hbox{$#1{#2#3}{\int}$}
     \vcenter{\hbox{$#2#3$}}\kern-.5\wd0}}

\usepackage[english]{babel}
\usepackage{xcolor}
\usepackage{enumerate}

\theoremstyle{plain}
\newtheorem{theo}{Theorem}[section]
\newtheorem{lem}[theo]{Lemma}
\newtheorem{prop}[theo]{Proposition}
\theoremstyle{definition}
\newtheorem{definition}[theo]{Definition}
\theoremstyle{remark}
\newtheorem{rem}[theo]{Remark}
\numberwithin{equation}{section}

\newcommand{\R}{\mathbb{R}}
\newcommand{\N}{\mathbb{N}}
\newcommand{\M}{\mathbb{M}}
\newcommand{\K}{\mathbb{K}}

\title{Size estimates for nanoplates
}

\author{Antonino Morassi\thanks{Dipartimento Politecnico di Ingegneria e Architettura,
Universit\`a degli Studi di Udine, via Cotonificio 114, 33100
Udine, Italy. E-mail: \textsf{antonino.morassi@uniud.it}}, \  Edi
Rosset\thanks{Dipartimento di Matematica e Geoscienze,
Universit\`a degli Studi di Trieste, via Valerio 12/1, 34127
Trieste, Italy. E-mail: \textsf{rossedi@units.it}}, \ Eva 
Sincich\thanks{Dipartimento di Matematica e Geoscienze,
	Universit\`a degli Studi di Trieste, via Valerio 12/1, 34127
	Trieste, Italy. E-mail: \textsf{esincich@units.it}} \
and Sergio
Vessella\thanks{Dipartimento di Matematica e Informatica ``Ulisse
Dini'', Universit\`a degli Studi di Firenze, Via Morgagni 67/a,
50134 Firenze, Italy. E-mail: \textsf{sergio.vessella@unifi.it}}}

\begin{document}

\maketitle

\noindent \textbf{Abstract.} 
We consider the problem of determining, within an elastic isotropic nanoplate in bending, the possible presence of an inclusion made of different elastic material. Under suitable a priori assumptions on the unknown inclusion, we provide quantitative upper and lower estimates for the area of the unknown defect in terms of the works exerted by the boundary data when the inclusion is present or absent.  

\medskip

\medskip

\noindent \textbf{Mathematical Subject Classifications (2010): } 35J30, 35R30,74K20


\medskip

\medskip

\noindent \textbf{Key words:}  Inverse problems, elastic nanoplates, size estimates, unique continuation.

\tableofcontents 

\section{Introduction} 
\label{sec:introduction}

Over the past three decades, micro- and nano-electromechanical systems (MEMS and NEMS) have found wide applications as sensors, actuators and for vibration control purposes \cite{Eom-PR-2011}. Due to their small size and the material properties, they possess superior mechanical, thermal and electrical performance compared to classical devices, allowing extreme miniaturisation, high reliability, low costs and reduced energy consumption for their operation. These indisputable advantages have favoured their rapid application in strategic areas, such as communications, biological technologies, mechanics and aerospace. 

Nanoplates are the core components of MEMS and NEMS, and their proper functionality is an essential requirement for the devices. The demand for higher performances and small sizes (typical size around $1 \div 10 \times 10^{-4}$ meters) have led to higher strain/stress states and very challenging operating conditions that can increase the probability of structural failure. Furthermore, defects such as cracks, internal voids, inhomogeneous material properties and abrasions can appear during the manufacturing process and can evolve during service, leading to the activation of mechanical device failure \cite{CYM-NLD-2017}, \cite{JSP-2009-JMS},  \cite{YZC-SV-2020}. 

For the reasons stated above, the problem of defect identification is attracting increasing attention from researchers interested in the behaviour of MEMS/NEMS devices. In this paper we consider the inverse problem of determining, within an isotropic elastic nanoplate subjected to static bending deformation, the possible presence of an inclusion made by different elastic material from a single measurement of boundary data.

Let us formulate the inverse problem. Let us consider a nanoplate in the referential configuration $\Omega \times [ -t/2, t/2  ]$, where $\Omega$ is a plane domain representing the middle surface of the nanoplate and $t$ is the uniform thickness, $t << diam(\Omega)$. Let $D$ be the subset corresponding to the unknown inclusion. It is well known that classical continuum mechanics, as a length-scale free theory, loses its predictive capacity for nanostructures since it is not able to take into account the presence of size effects in the material response. Here, we shall adopt the simplified strain gradient theory proposed by Lam et al. \cite{Lam2003} to model the mechanical behavior of the material in infinitesimal deformation. Under the kinematic assumptions of Kirchhoff-Love's plate theory, the statical equilibrium problem of the nanoplate loaded at the boundary and under vanishing body forces is described by the following Neumann boundary value problem \cite{KMZ-MMAS-2022}
\begin{equation}
	\label{eq:introduzione-PDE}
	(M_{\alpha \beta} +  \overline{M}_{\alpha \beta \gamma, \gamma}^h   )_{,\alpha \beta}=0 \quad \hbox{in } \Omega,
\end{equation}
\begin{multline}
	\label{eq:introduzione-bc1}
	(
	M_{\alpha \beta} +  \overline{M}_{\alpha \beta \gamma, \gamma}^h
	)_{,\alpha} n_\beta
	+
	(
	(
	M_{\alpha \beta} +  \overline{M}_{\alpha \beta \gamma, \gamma}^h
	)n_\alpha {\tau}_\beta	
	)_{,s}
	+
	(
	\overline{M}_{\alpha \beta \gamma}^h {\tau}_\alpha {\tau}_\beta n_\gamma
	)_{,ss} -
	\\
	-
	(
	\overline{M}_{\alpha \beta \gamma}^h n_\gamma
	(  {\tau}_{\alpha,s} {\tau}_\beta - n_{\alpha,s} n_\beta )
	)	
	_{,s}
	=
	- \widehat{V}  \qquad \hbox{on } \partial \Omega,
\end{multline}
\begin{multline}
	\label{eq:introduzione-bc2}
	(
	M_{\alpha \beta} +  \overline{M}_{\alpha \beta \gamma, \gamma}^h
	)n_\alpha n_\beta 
	+
	(
	\overline{M}_{\alpha \beta \gamma}^h n_\gamma ({\tau}_\alpha n_\beta + {\tau}_\beta n_\alpha )
	)_{,s}
	-
	\\
	- \overline{M}_{\alpha \beta \gamma}^h n_\gamma
	(n_{\alpha,s} {\tau}_\beta)
	=\widehat{M}_n  \qquad \hbox{on } \partial \Omega,
\end{multline}
\begin{equation}
	\label{eq:introduzione-bc3}
	\overline{M}_{\alpha \beta \gamma}^h n_\alpha n_\beta n_\gamma
	=-\widehat{M}_n^h  \qquad \hbox{on } \partial \Omega.
\end{equation}
The functions $M_{\alpha \beta}= M_{\alpha \beta}(u)$, $\overline{M}_{\alpha \beta \gamma}^h=\overline{M}_{\alpha \beta \gamma}^h(u)$, $\alpha, \beta, \gamma=1,2$, in the above equations are the cartesian components of the couple tensor $M=(M_{\alpha \beta})$ and the high-order couple tensor $\overline{M}^h=(\overline{M}_{\alpha \beta \gamma}^h)$, respectively, corresponding to the transverse displacement $u$. The constitutive equations of $M$ and $\overline{M}^h$ are as follows
\begin{equation}
	\label{eq:introduzione-Mcompatto}
	M(u)= (\chi_{\Omega \setminus D} (\mathbb{P}+\mathbb{P}^h   ) + \chi_D ( \widetilde{\mathbb{P}} + \widetilde{\mathbb{P}}^h   ))D^2u,
\end{equation}
\begin{equation}
	\label{eq:introduzione-M^hcompatto}
	\overline{M}^h(u)= (\chi_{\Omega \setminus D} \mathbb{Q}  + \chi_D  \widetilde{\mathbb{Q}})D^3u,
\end{equation}
where $\mathbb{P}$, $\mathbb{P}^h$, $\mathbb{Q}$ $\widetilde{\mathbb{P}}$, $\widetilde{\mathbb{P}}^h$, $\widetilde{\mathbb{Q}}$ are the tensors expressing the response of the material and are defined in detail in Section \ref{sec:results}. 
The vectors $\tau$ and $n$ are the unit tangent and the unit outer normal to $\partial \Omega$, and $s$ is an arclength chosen on $\partial \Omega$. The loads acting on $\partial \Omega$ include the shear force $\widehat{V}$, the bending moment $\widehat{M}_n$ and the high-order bending moment $\widehat{M}_n^h$. A full description of the mechanical nanoplate model and the main properties of the direct problem can be found in Section \ref{sec:direct-pbm}, to which we refer for .

The first question to be asked in approaching our inverse problem is the question of uniqueness. In particular: \textit{does a single boundary measurement of  Neumann data $\{\widehat{V}, \widehat{M}_n, \widehat{M}_n^h\}$ and Dirichlet data $\{u, u_{,n}, u_{,nn}\}$ uniquely determine the unknown inclusion $D$?} In spite of the simplicity with which it is formulated, this inverse problem is extremely difficult and even in the simpler context of electrical impedance tomography, which involves a second-order elliptic equation, a general uniqueness result is missing. We refer tox  \cite{AGL2015}, \cite{A98}, \cite{CYM-NLD-2017} for an up-to-date overview and an extensive reference list.

In the present note, we discuss another direction of research. In fact, instead of determining the exact shape and location of $D$, we evaluate its size in terms of the data. More precisely we provide quantitative estimate on the area of the unknown inclusion in terms of the quantities  
\begin{equation}
	\label{eq:11-1-intro}
	W=L(u) = -\int_{\partial \Omega}  \widehat{V} u + \widehat{M}_n u,_{n} + \widehat{M}^h_n u,_{nn},
\end{equation}
\begin{equation}
	\label{eq:11-2-intro}
	W_0=L(u_0) = -\int_{\partial \Omega}  \widehat{V} u_0 + \widehat{M}_n u_{0,n} + \widehat{M}^h_n u_{0,nn},
\end{equation}
which represent the works exerted by the boundary data when the inclusion $D$ is present or absent, respectively. Here $u_0$ is the transverse displacement of the reference nanoplate without inclusion, namely $u_0$ satisfies \eqref{eq:introduzione-PDE}-\eqref{eq:introduzione-bc3} when $D$ is the empty set. 

In order to treat the inverse problem, we first need to analyze the direct one. In Subsection \ref{sec:direct-pbm}, we collect some previous results contained in \cite{MRSV2022}, concerning the well posedness of the direct problem (see Theorem \ref{theo:theorem-5-1}) and, for the case in which the inclusion is absent, the $H^4$ regularity up to the boundary of the solution of the Neumann problem  \eqref{eq:introduzione-PDE}-\eqref{eq:introduzione-bc3}  (see Theorem \ref{theo:5bis-1}) and the $H^6$ regularity in the interior for solutions of the underlying equation \eqref{eq:introduzione-PDE} (see Theorem \ref{theo:5ter-1}), under suitable regularity assumptions on the coefficients and on the boundary of $\Omega$.

In Subsection \ref{sec:inverse-pbm} we rigorously formulate the inverse problem and state our main a-priori assumptions.  In Subsection \ref{sec:main-results} we present our main results that can be summarized as follows. 
\begin{enumerate}
\item [i)] In Theorem \ref{theo:13-1} we consider the case when $D$ is a general measurable set compactly contained in $\Omega$ and we provide the following lower bound of its size
\begin{equation}\label{intro:LB}
|D|\ge C\left| \frac{W-W_0}{W_0}\right|,
\end{equation}
where $C$ is estimated in terms of the a priori data. The main idea underlying this kind of estimates is that the integral 
\begin{equation}\label{intro:energyestimate}
\int_{D} |D^2 u_0|^2 + |D^3 u_0|^2\ \ \mbox{is comparable to}\ \ |W_0-W| \ . 
\end{equation}
The above behavior follows from energy estimates for the Neumann problem \eqref{eq:introduzione-PDE}-\eqref{eq:introduzione-bc3} both when the inclusion is present and when it is absent (see Lemma \ref{lem:16.1} for a precise statement). By using interior regularity estimate for the sixth order elliptic equation we can control from below the size of $D$ in terms of the integral in \eqref{intro:energyestimate} and achieve the desired bound \eqref{intro:LB}. 
\item [ii)] In Theorem \ref{theo:14-1} we prove an upper bound for the size of $D$ under the following so-called {\it fatness condition} on $D$ itself. Namely, given $h>0$ and denoting $D_h=\{x\in D: \mbox{dist}(x,\partial D)>h \}$, if we assume 
\begin{equation}\label{intro:fatcond}
|D_h|\ge \frac{1}{2} |D|,
\end{equation}
then we have 
\begin{equation}\label{UFat}
|D|\le C\bigg| \frac{W-W_0}{W_0}\bigg|,
\end{equation}
where $C$ is estimated in terms of the a priori data.  Although in order to obtain \eqref{UFat} we still make use of \eqref{intro:energyestimate}, in this step a deeper analysis is required since we need to estimate $\int_{D} |D^2 u_0|^2 + |D^3 u_0|^2$ from below, but in general $D^2 u_0$ and $D^3 u_0$ may vanish at interior points.  In this respect, we prove the following unique continuation result known as {\it Lipschitz propagation of smallness} (see Proposition \ref{LPS} for a precise statement).

 There exists $\chi>1$ depending on the a priori data such that for every $\rho>0$ and every $x\in \Omega$ such that $\mbox{dist}(x,\partial \Omega)>\chi\rho$ we have 
\begin{equation}\label{intro:LPS}
\int_{B_{\rho}(x)}|D^2 u_0|^2 \ge C \int_{\Omega}|D^2 u_0|^2, 
\end{equation}
where $C>0$ is estimated in terms of $\rho$ and the a priori data. 
From such an estimate, inequality \eqref{UFat} follows by covering  $D_h$ by non overlapping squares of side $\epsilon = O(h)$. 

\item [iii)] In Theorem \ref{theo:15-1}, we remove the fatness condition on $D$ (compactly contained in $\Omega$) and we state an upper bound for the size of $D$ of the following form 
\begin{equation}\label{Ugeneralcase}
|D|\le C\bigg| \frac{W-W_0}{W_0}\bigg|^{1/p},
\end{equation}
where $C>0, p>1$ are estimated in terms of the a priori data. 

In this case, in contrast with the `` fat" one, we need to introduced a further sophisticated argument arising in the theory of Muckenhoupt weight (see Proposition \ref{prop:Ap}).
By combining the $A_p$-estimates with a covering argument and \eqref{intro:energyestimate} we end up with the desired estimate.
 \end{enumerate}
Let us also recall that the prototype of this class of inverse problems is the determination of the size of an inclusion within an electrostatic conductor (\cite{AR1998,ARS2000,CV2003,KSS97}) and that such an issue has been extended to more complicated equations and systems (\cite{AMR2002,BMRTV2019,DLMRVW,Ik93, MRV2007,MRV2009,MRV2013,MRV2018}) . However, although our strategy belongs to the ones adopted in this line of research, the treatment of a higher order underlying partial differential equations has required the development of new tools. 

Indeed in Section \ref{sec:DIandTS} we present some new estimate of unique continuation in the form of a doubling inequality and a three sphere inequality for the Hessian of the solution of the unperturbed nanoplate. The iterated use of such three sphere inequalities allows us to obtain \eqref{intro:LPS}. Moreover, we also provide a global doubling inequality expressed in terms  of the known boundary data (see Proposition \ref{DoublingData}) which establishes a bridge with the theory of Muckenhoupt weight.

\section{Notation} 
\label{sec:notation}
\noindent 
Let $P=(x_1(P), x_2(P))$ be a point of $\R^2$. We shall denote by
$B_r(P)$ the disk in $\R^2$ of radius $r$ and center $P$ and by
$R_{a,b}(P)$ the rectangle of center $P$ and sides parallel to the
coordinate axes, of length $2a$ and $2b$, namely
\begin{equation}
	\label{eq:rectangle}
	R_{a,b}(P)=\{x=(x_1,x_2)\ |\ |x_1-x_1(P)|<a,\ |x_2-x_2(P)|<b \}.     
\end{equation}

\begin{definition}
	\label{def:reg_bordo} (${C}^{k,\alpha}$ regularity)
	Let $\Omega$ be a bounded domain in ${\R}^{2}$. Given $k,\alpha$,
	with $k\in\N$, $k\geq 1$, $0<\alpha\leq 1$, we say that a portion $S$ of
	$\partial \Omega$ is of \textit{class ${C}^{k,\alpha}$ with
		constants $r_{0}$, $M_{0}>0$}, if, for any $P \in S$, there exists
	a rigid transformation of coordinates under which we have $P=0$
	and
	\begin{equation*}
		\Omega \cap R_{r_0,2M_0r_0}=\{x \in R_{r_0,2M_0r_0} \quad | \quad
		x_{2}>g(x_1)
		\},
	\end{equation*}
	where $g$ is a ${C}^{k,\alpha}$ function on $[-r_0,r_0]$
	satisfying
	\begin{equation*}
		g(0)=g'(0)=0,
	\end{equation*}
	\begin{equation*}
		\|g\|_{{C}^{k,\alpha}([-r_0,r_0])} \leq M_0r_0,
	\end{equation*}
	where
	\begin{equation*}
		\|g\|_{{C}^{k,\alpha}([-r_0,r_0])} = \sum_{i=0}^k
		r_0^i\sup_{[-r_0,r_0]}|g^{(i)}|+r_0^{k+\alpha}|g|_{k,\alpha},
	\end{equation*}
	
	\begin{equation*}
		|g|_{k,\alpha}= \sup_ {\overset{\scriptstyle t,s\in
				[-r_0,r_0]}{\scriptstyle t\neq s}}\frac{|g^{(k)}(t) -
			g^{(k)}(s)|}{|t-s|^\alpha}.
	\end{equation*}
\end{definition}

We use the convention to normalize all norms in such a way that
their terms are dimensionally homogeneous and coincide with the
standard definition when the dimensional parameter equals one. For
instance, given a function $u:\Omega \rightarrow \R$ we denote 

\begin{equation}
	\label{eq:1.1.1}
	\|u\|_{H^k(\Omega)}=r_0^{-1} \left ( \sum_{i=0}^k r_0^{2k}\int_\Omega |D^ku|^2
	\right )^{\frac{1}{2}}
\end{equation}
where 
\begin{equation}
\label{eq:1.1.2}
\int_{\Omega} |D^k u |^2=\int_{\Omega}\sum_{|\alpha|=k} |D^\alpha u |^2
\end{equation}


%
and so on for boundary and trace norms.

For any $h>0$ we set 

$$\Omega_{h}=\{x\in \Omega \ : \ \mbox{dist}(x,\partial \Omega)>h \}.$$

Given a bounded domain $\Omega$ in $\R^2$ such that $\partial
\Omega$ is of class $C^{k,\alpha}$, with $k\geq 1$, we consider as
positive the orientation of the boundary induced by the outer unit
normal $n$ in the following sense. Given a point
$P\in\partial\Omega$, let us denote by $\tau=\tau(P)$ the unit
tangent at the boundary in $P$ obtained by applying to $n$ a
counterclockwise rotation of angle $\frac{\pi}{2}$, that is
\begin{equation}
	\label{eq:2.tangent}
	\tau=e_3 \times n,
\end{equation}
where $\times$ denotes the vector product in $\R^3$ and $\{e_1,
e_2, e_3\}$ is the canonical basis in $\R^3$.

Given any connected component $\cal C$ of $\partial \Omega$ and
fixed a point $P_0\in\cal C$, let us define as positive the
orientation of $\cal C$ associated to an arclength
parameterization $\psi(s)=(x_1(s), x_2(s))$, $s \in [0, l(\cal
C)]$, such that $\psi(0)=P_0$ and $\psi'(s)=\tau(\psi(s))$. Here
$l(\cal C)$ denotes the length of $\cal C$.

Throughout the paper, we denote by $w,_\alpha$, $\alpha=1,2$,
$w,_s$, and $w,_n$ the derivatives of a function $w$ with respect
to the $x_\alpha$ variable, to the arclength $s$ and to the normal
direction $n$, respectively, and similarly for higher order
derivatives.

We denote by $\M^{2}, \M^{3}$ the Banach spaces of second order and the third order tensors and by 
$\widehat{\M}^{2}, \widehat{\M}^{3}$ the corresponding subspaces of tensors having components invariant with respect to permutations of the indexes.

Let ${\cal L} (X, Y)$ be the space of bounded linear
operators between Banach spaces $X$ and $Y$. Given $\K\in{\cal L} ({\M}^{2},{\M}^{2})$ and $A,B\in \M^{2}$, we use the following notation 
\begin{equation}
	\label{eq:2.notation_1}
	({\K}A)_{ij} = \sum_{l,m=1}^{2} K_{ijlm}A_{lm},
\end{equation}
\begin{equation}
	\label{eq:2.notation_2}
	A \cdot B = \sum_{i,j=1}^{2} A_{ij}B_{ij},
\end{equation}

Similarly, given $\K\in{\cal L} ({\M}^{3},
{\M}^{3})$ and $A,B\in \M^{3}$, we denote
\begin{equation}
	\label{eq:2.notation_1bis}
	({\K}A)_{ijk} = \sum_{l,m,n=1}^{2} K_{ijklmn}A_{lmn},
\end{equation}
\begin{equation}
	\label{eq:2.notation_2bis}
	A \cdot B = \sum_{i,j,k=1}^{2} A_{ijk}B_{ijk},
\end{equation}

Moroever, for any $A\in \M^{n}$, with $ n=2,3$, we shall denote 
\begin{equation}
	\label{eq:2notation_3}
	|A|= (A \cdot A)^{\frac {1} {2}}.
\end{equation}


%
%
%
%
%

The linear space of the infinitesimal rigid displacements is defined as
\begin{equation}
	\label{eq:def_rig_displ-1}
	{\cal R}_2 = \left \{
	r(x) = c + Wx, \ c \in \R^2, \ W \in \M^2, \ W+W^T=0
	\right \}.
\end{equation}
%

\noindent Throughout the paper, summation over repeated indexes is assumed.

\section{Size estimates results} 
\label{sec:results}

\subsection{The direct problem} 
\label{sec:direct-pbm}

Let us consider a nanoplate $\Omega \times  \left ( -\frac{t}{2}, \frac{t}{2}   \right )$ with middle surface $\Omega$ represented by a bounded domain of $\R^2$ and having constant thickness $t$, $t << diam (\Omega)$. We assume that the boundary $\partial \Omega$ of $\Omega$ is of class $C^{2,1}$ with constants $r_0$, $M_0$ and that 
\begin{equation}
	\label{eq:1-1}
	|\Omega| \leq M_1 r_0^2,
\end{equation}
where $M_1$ is a positive constant.

Within the kinematic framework of the Kirchhoff-Love theory in infinitesimal deformation, the statical equilibrium problem of the nanoplate loaded at the boundary and under vanishing body forces is described by the following Neumann boundary value problem \cite{KMZ-MMAS-2022}:
\begin{equation}
	\label{eq:1-2}
	(M_{\alpha \beta} +  \overline{M}_{\alpha \beta \gamma, \gamma}^h   )_{,\alpha \beta}=0 \quad \hbox{in } \Omega,
\end{equation}
\begin{multline}
	\label{eq:1-3}
	(
	M_{\alpha \beta} +  \overline{M}_{\alpha \beta \gamma, \gamma}^h
	)_{,\alpha} n_\beta
	+
	(
	(
	M_{\alpha \beta} +  \overline{M}_{\alpha \beta \gamma, \gamma}^h
	)n_\alpha {\tau}_\beta	
	)_{,s}
	+
	(
	\overline{M}_{\alpha \beta \gamma}^h {\tau}_\alpha {\tau}_\beta n_\gamma
	)_{,ss} -
	\\
	-
	(
	\overline{M}_{\alpha \beta \gamma}^h n_\gamma
	(  {\tau}_{\alpha,s} {\tau}_\beta - n_{\alpha,s} n_\beta )
	)	
	_{,s}
	=
	- \widehat{V}  \qquad \hbox{on } \partial \Omega,
\end{multline}
\begin{multline}
	\label{eq:1-4}
	(
	M_{\alpha \beta} +  \overline{M}_{\alpha \beta \gamma, \gamma}^h
	)n_\alpha n_\beta 
	+
	(
	\overline{M}_{\alpha \beta \gamma}^h n_\gamma ({\tau}_\alpha n_\beta + {\tau}_\beta n_\alpha )
	)_{,s}
	-
	\\
	- \overline{M}_{\alpha \beta \gamma}^h n_\gamma
	(n_{\alpha,s} {\tau}_\beta)
	=\widehat{M}_n  \qquad \hbox{on } \partial \Omega,
\end{multline}
\begin{equation}
	\label{eq:1-5}
	\overline{M}_{\alpha \beta \gamma}^h n_\alpha n_\beta n_\gamma
	=-\widehat{M}_n^h  \qquad \hbox{on } \partial \Omega.
\end{equation}
The functions $M_{\alpha \beta}= M_{\alpha \beta}(u)$, $\overline{M}_{\alpha \beta \gamma}^h=\overline{M}_{\alpha \beta \gamma}^h(u)$, $\alpha, \beta, \gamma=1,2$, in the above equations are the cartesian components of the couple tensor $M=(M_{\alpha \beta})$ and the high-order couple tensor $\overline{M}^h=(\overline{M}_{\alpha \beta \gamma}^h)$, respectively, corresponding to the transverse displacement $u(x_1,x_2)$, $u : \Omega \rightarrow \R$, of the point $(x_1,x_2)=x$ belonging to the middle surface of the nanoplate. To simplify the presentation, the dependence of these quantities on $u$ is not explicitly indicated in \eqref{eq:1-2}--\eqref{eq:1-5} and in what follows.

We assume that the functions $M_{\alpha\beta}$ can be expressed as
\begin{equation}
	\label{eq:2-1}
	M_{\alpha\beta} = -( P_{\alpha \beta \gamma \delta} + P_{\alpha \beta \gamma \delta}^h  ) u_{,\gamma \delta} \quad(M=-(\mathbb{P}+\mathbb{P}^h)D^2u)),
\end{equation}
where the fourth order tensors $\mathbb{P} = \mathbb{P}(x) \in L^\infty (\Omega,  \mathcal{L}(\widehat{\M}^2, \widehat{\M}^2) )$, $\mathbb{P}^h = \mathbb{P}^h(x) \in L^\infty (\Omega,  \mathcal{L}( \widehat{\M}^2, \widehat{\M}^2)   )$ are assumed to satisfy the symmetry conditions
\begin{equation}
	\label{eq:3-1}
	\mathbb{P}A\cdot B=\mathbb{P}B\cdot A, \quad \hbox{a.e. in } \ \Omega,
\end{equation}
\begin{equation}
	\label{eq:3-2}
	\mathbb{P}^h A\cdot B=\mathbb{P}^h B\cdot A, \quad \hbox{a.e. in } \ \Omega,
\end{equation}
for every $A,B\in \widehat{\M}^2$, and the strong convexity condition
\begin{equation}
	\label{eq:3-3}
	(\mathbb{P}+\mathbb{P}^h)A\cdot A \geq t^3\xi_ {\mathbb{P}} |A|^2, \quad \hbox{a.e. in } \ \Omega,
\end{equation}
for every $A\in \widehat{\M}^2$, where $\xi_ {\mathbb{P}}$ is a positive constant.

Concerning the functions $\overline{M}_{ijk}^h$ ($i,j,k=1,2$), we assume that they can be expressed as
\begin{equation}
	\label{eq:3-5}
	\overline{M}_{ijk}^h= Q_{ijklmn}u_{,lmn}\quad (\overline{M}^h=\mathbb{Q} D^3u),
\end{equation}
where $Q_{ijklmn}$ are the cartesian components of the sixth order tensor $\mathbb{Q}=\mathbb{Q}(x) \in L^\infty( \Omega,   \mathcal{L}(\widehat{\M}^3, \widehat{\M}^3   )    )$, and $\mathbb{Q}$ is assumed to satisfy the symmetry conditions 
\begin{equation}
	\label{eq:3-6}
	\mathbb{Q}A\cdot B = \mathbb{Q}B\cdot A, \quad \hbox{a.e. in } \ \Omega,
\end{equation}
for every $A,B\in \widehat{\M}^3$, and the strong convexity condition
\begin{equation}
	\label{eq:3-7}
	\mathbb{Q}A\cdot A \geq t^5\xi_ { \mathbb{Q}} | A|^2, \quad \hbox{a.e. in } \ \Omega,
\end{equation}
for every $A\in \widehat{\M}^3$, where $\xi_ { \mathbb{Q}}$ is a  positive constant.

On the loading data $\widehat{V}$ (shear force), $\widehat{M}_n$ (bending moment) and $\widehat{M}_n^h$ (high-order bending moment) appearing in the boundary equilibrium equations \eqref{eq:1-3}--\eqref{eq:1-5}, we require the following regularity conditions
\begin{equation}
	\label{eq:4-1}
	\widehat{V} \in H^{ - 5/2  }(\partial \Omega), \quad 
	\widehat{M}_n \in H^{ - 3/2  }(\partial \Omega), \quad 
	\widehat{M}_n^h \in H^{ - 1/2  }(\partial \Omega)
\end{equation}
and the compatibility conditions (see \cite{KMZ-MMAS-2022})
\begin{equation}
	\label{eq:4-2}
	\int_{\partial \Omega} \widehat{V} =0\ , \  \int_{\partial \Omega}  \widehat{V} x_1 +  \widehat{M}_n n_1=0\ , \ \int_{\partial \Omega}  \widehat{V} x_2 +  \widehat{M}_n n_2=0.\
\end{equation} 
The weak formulation of the Neumann problem \eqref{eq:1-2}--\eqref{eq:1-5}, with loading data satisfying \eqref{eq:4-1} and \eqref{eq:4-2}, consists in determining a function $u \in H^3(\Omega)$ (weak solution) such that
\begin{equation}
	\label{eq:4bis-1}
	a(u,w)=L(w) \ , \ \ \ \mbox{for every} \ w\in H^3(\Omega),
\end{equation}
where
\begin{multline}
	\label{eq:4bis-2}
	a(u,w)=\int_{\Omega} -M_{\alpha\beta}(u)w,_{\alpha \beta} + \overline{M}_{\alpha \beta \gamma}^h(u)w,_{\alpha \beta \gamma} =\\
	=\int_{\Omega}(\mathbb{P}+\mathbb{P}^h)D^2uD^2w+\mathbb{Q}D^3uD^3w,
\end{multline}
\begin{equation}
	\label{eq:4bis-3}
	L(w)=-\int_{\partial \Omega}  \widehat{V} w + \widehat{M}_n w,_{n} + \widehat{M}^h_n w,_{nn} .
\end{equation}
Finally, in order to identify a unique solution, we assume the following normalization conditions
\begin{equation}
	\label{eq:4-3}
	\int_{\Omega} u=0 \ , \ \ \int_{\Omega} u_{,\alpha}=0 \ , \ \ \alpha=1,2\ .
\end{equation}
We are now in position to state the existence, uniqueness and regularity
results useful in our analysis. Details of the proofs can be found in \cite{KMZ-MMAS-2022, MRSV2022}.
\begin{theo} [\bf Existence, uniqueness and $H^3$-regularity]
	\label{theo:theorem-5-1}
	Let $\Omega$ be a bounded domain in $\mathbb{R}^2$ with boundary $\partial \Omega$ of class $C^{2,1}$ with constant $r_0,M_0$.
	Let the tensors $\mathbb{P}$, $\mathbb{P}^h \in L^\infty( \Omega, \mathcal{L}(\widehat{\M}^2, \widehat{\M}^2   )  )$ and $\mathbb{Q} \in L^\infty( \Omega, \mathcal{L}(\widehat{\M}^3, \widehat{\M}^3   )  )$ satisfy the symmetry conditions \eqref{eq:3-1}, \eqref{eq:3-2}, \eqref{eq:3-6} and the strong convexity conditions \eqref{eq:3-3}, \eqref{eq:3-7}, respectively. 
	Let the data $\widehat{V}, \widehat{M}_n, \widehat{M}_n^h $ as in \eqref{eq:4-1} and satisfying the compatibility conditions \eqref{eq:4-2}.
	
	 The Neumann problem \eqref{eq:1-2}--\eqref{eq:1-5} admits a unique weak solution $u \in H^3(\Omega)$ satisfying \eqref{eq:4-3} and, moreover, 
	\begin{equation}
		\label{eq:5-1}
		\|u\|_{H^3(\Omega)}\le C \left(\|\widehat{V}\|_{H^{ - 5/2  }(\partial \Omega)} + r_0^{-1}\| \widehat{M}_n\|_{H^{ - 3/2  }(\partial \Omega)} + r_0^{-2}\|  \widehat{M}_n^h\|_{H^{ - 1/2  }(\partial \Omega)} \right)
	\end{equation}
	where the constant $C>0$ only depends on $\frac{t}{r_0}$, $M_0$, $M_1$, $\xi_{\mathbb{P}}$, $\xi_{\mathbb Q}$.
\end{theo}

We conclude this section with a global and an improved interior regularity result.
\begin{theo}[\bf Global $H^4$-regularity]
	\label{theo:5bis-1}
	Let $\Omega$ be a bounded domain in $\R^2$ with boundary
	$\partial \Omega$ of class $C^{3,1}$ with constants $r_0$, $M_0$, and satisfying \eqref{eq:1-1}. 	Let the tensors $\mathbb{P}$, $\mathbb{P}^h \in C^{0,1}( \overline{\Omega}, \mathcal{L}(\widehat{\M}^2, \widehat{\M}^2   )  )$ and $\mathbb{Q} \in  C^{0,1}( \overline{\Omega}, \mathcal{L}(\widehat{\M}^3, \widehat{\M}^3   )  )$ satisfy the symmetry conditions \eqref{eq:3-1}, \eqref{eq:3-2}, \eqref{eq:3-6} and the strong convexity conditions \eqref{eq:3-3}, \eqref{eq:3-7}, respectively. Let $u \in H^3(\Omega)$ be the weak solution of the Neumann problem
	\eqref{eq:1-2}--\eqref{eq:1-5} satisfying \eqref{eq:4-3}, where $\widehat{V} \in H^{ - 3/2  }(\partial \Omega), \quad 
	\widehat{M}_n \in H^{ - 1/2  }(\partial \Omega), \quad \widehat{M}_n^h \in H^{  1/2  }(\partial \Omega)$ are such that the compatibility conditions \eqref{eq:4-2} are satisfied. 
		
	Then $u \in H^4(\Omega)$ and
	\begin{equation}
		\label{eq:5bis-1}
		\| u \|_{H^4(\Omega)} \leq C \left(\|\widehat{V}\|_{H^{ - 3/2  }(\partial \Omega)} + r_0^{-1}\| \widehat{M}_n\|_{H^{ - 1/2  }(\partial \Omega)} + r_0^{-2}\|  \widehat{M}_n^h\|_{H^{  1/2  }(\partial \Omega)} \right),
	\end{equation}
	where the constant $C>0$ only depends on $\frac{t}{r_0}$, $M_0$, $M_1$, $\xi_{\mathbb P}$, $\xi_{\mathbb Q}$, $\|\mathbb P\|_{C^{0,1}( \overline{\Omega})}$, $\|\mathbb P^h\|_{C^{0,1}( \overline{\Omega})}$, $\|\mathbb Q\|_{C^{0,1}( \overline{\Omega})}$.
\end{theo}
\begin{theo} [{\bf Improved interior regularity}]
	\label{theo:5ter-1}
	Let $B_{\sigma}$ be an open ball in $\R^2$ centered
	at the origin and with radius $\sigma$. Let $u \in H^3(B_{\sigma})$ be such that
	\begin{equation}
		\label{eq:5ter-1}
		a(u, \varphi) = 0		   \qquad \hbox{for every }
		\varphi \in H^3_0(B_{\sigma}),
	\end{equation}
	with
	\begin{equation}
		\label{eq:5ter-2}
		a(u, \varphi) = \int_{B_{\sigma}}  (\mathbb P + \mathbb P^h) D^2u \cdot D^2 \varphi + \mathbb Q D^3 u \cdot D^3 \varphi,
	\end{equation}
	where the tensors $\mathbb P, \mathbb P^h \in C^{1,1} (\overline{B_{\sigma}}, \mathcal{L}(\widehat{\M}^2, \widehat{\M}^2   )    )$, $\mathbb Q \in C^{2,1} (\overline{B_{\sigma}}, \mathcal{L}(\widehat{\M}^3, \widehat{\M}^3   )  )$
	satisfy the symmetry conditions \eqref{eq:3-1}, \eqref{eq:3-2}, \eqref{eq:3-6} and the strong convexity conditions \eqref{eq:3-3}, \eqref{eq:3-7}, respectively. 
	
	Then $u \in H^6( B_{  \frac{\sigma}{8}  })$ and we have
	\begin{equation}
		\label{eq:5ter-3}
		\|u\|_{ H^6 (B_{\frac{\sigma}{8}})} \leq C  \|u\|_{ H^3 (B_{\sigma})},
	\end{equation}
	where $C>0$ only depends on $t$, $\xi_{\mathbb P}$, $\xi_{\mathbb Q}$, $\|\mathbb P\|_{C^{1,1}( \overline{B_\sigma})}$, $\|\mathbb P^h\|_{C^{1,1}( \overline{B_\sigma})}$, $\|\mathbb Q\|_{C^{2,1}( \overline{B_\sigma})}$.
\end{theo}

\subsection{Formulation of the inverse problem} 
\label{sec:inverse-pbm}

We consider a nanoplate $\Omega \times \left ( -\frac{t}{2},   \frac{t}{2}    \right )$ inside which a possible inclusion $D \times \left ( -\frac{t}{2},   \frac{t}{2}    \right )$ is present. The inclusion $D$ is a measurable, possibly disconnected, subset of $\Omega$. 

Let us consider elasticity tensors  $\mathbb{P}$, $\widetilde{\mathbb{P}}$, $\mathbb{P}^h$, $\widetilde{\mathbb{P}}^h \in L^\infty( \Omega, \mathcal{L}(\widehat{\M}^2, \widehat{\M}^2   )  )$ and  $\mathbb{Q}$, $\widetilde{\mathbb{Q}} \in L^\infty( \Omega, \mathcal{L}(\widehat{\M}^3, \widehat{\M}^3   )  )$ satisfying the symmetry conditions \eqref{eq:3-1}, \eqref{eq:3-2} and \eqref{eq:3-6}, respectively. 

We shall make the following a-priori assumptions on the elasticity tensors.

i) \textit{Isotropy for $\mathbb{P}$, $\mathbb{P}^h$, $\mathbb{Q}$.}

The cartesian components of $\mathbb{P}$, $\mathbb{P}^h$, $\mathbb{Q}$ are given by
\begin{equation}
	\label{eq:6-1}
	P_{\alpha \beta \gamma \delta}= B((1-\nu)\delta_{\alpha \gamma} \delta_{\beta \delta} + \nu \delta_{\alpha \beta} \delta_{\gamma\delta}
	),
\end{equation}
\begin{equation}
	\label{eq:6-2}
	P_{\alpha \beta \gamma \delta}^h= (2a_2+5a_1)\delta_{\alpha \gamma} \delta_{\beta \delta} + (-a_1 -a_2 +a_0) \delta_{\alpha \beta} \delta_{\gamma\delta},
\end{equation}
\begin{multline}
	\label{eq:7-0}
	Q_{ijklmn}= \frac{1}{3}(b_0 -3b_1)\delta_{ij}\delta_{kn}\delta_{lm}+
	\\
	+ \frac{1}{6}(b_0 -3b_1)
	(
	\delta_{ik} ( \delta_{jl} \delta_{mn} + \delta_{jm}\delta_{ln}   )
	+
	\delta_{jk} ( \delta_{il} \delta_{mn} + \delta_{im}\delta_{ln}   )
	)
	+
	Q_8 
	(
	\delta_{kn} ( \delta_{il}\delta_{jm} +\delta_{im}\delta_{jl}    )
	)
	+
	\\
	+
	Q_9
	(
	\delta_{jn} ( \delta_{il} \delta_{km} + \delta_{im}\delta_{kl}   )
	+
	\delta_{in} ( \delta_{jl} \delta_{km} + \delta_{jm}\delta_{kl}   )
	),	
\end{multline}
where $2(Q_8+2Q_9)=5b_1$.

The bending stiffness (per unit length) $B=B(x)$ is given by the function
\begin{equation}
	\label{eq:7-1}
	B(x) = \frac{t^3 E(x)}{12(1-\nu^2(x))}, \quad \hbox{a.e. in } \Omega,
\end{equation}
where the Young's modulus $E$ and the Poisson's coefficient $\nu$ of the material can be written in terms of the Lamé moduli $\mu$ and $\lambda$ as follows
\begin{equation}
	\label{eq:7-2}
	E(x) = \frac{\mu(x)(2\mu(x)+3\lambda(x))}{ \mu(x) +\lambda(x)  },
	\quad 
	\nu(x) = \frac{\lambda(x)}{2(\mu(x)+\lambda(x))}.
\end{equation}
The coefficients $a_i(x)$, $i=0,1,2$, are given by (see \cite{KMZ-MMAS-2022})
\begin{equation}
	\label{eq:7-3}
	a_0(x)=2\mu(x)t\mathit{l}_0^2, \quad a_1(x) = \frac{2}{15}\mu(x)t  \mathit{l}_1^2, \quad a_2(x) = \mu(x)t\mathit{l}_2^2 \quad \hbox{a.e. in } \ \Omega,
\end{equation}
where the material length scale parameters $\mathit{l}_i$ are assumed to be positive constants. We denote

\begin{equation}
	\label{eq:l}
	l=\min\{l_0,l_1, l_2\}.
\end{equation}

The coefficients $b_i(x)$, $i=0,1$, are given by 
\begin{equation}
	\label{eq:7-4}
	b_0(x)=2\mu(x)\frac{t^3}{12}\mathit{l}_0^2, \quad 
	b_1(x)=\frac{2}{5}\mu(x)\frac{t^3}{12}\mathit{l}_1^2
	\quad \hbox{a.e. in } \ \Omega.
\end{equation}

ii) \textit{Strong convexity for $\mathbb{P}+\mathbb{P}^h$, $\mathbb{Q}$.}

We assume the following ellipticity conditions on $\mu$ and $\lambda$:
\begin{equation}
	\label{eq:8-1}
	\mu(x) \geq \alpha_0 >0, \quad 2\mu(x)+3\lambda(x) \geq \gamma_0 >0 \quad \hbox{a.e. in } \ \Omega,
\end{equation}
where $\alpha_0$, $\gamma_0$ are positive constants. By \eqref{eq:7-3}, \eqref{eq:7-4} and \eqref{eq:8-1} we also have
\begin{equation}
	\label{eq:8-2}
	a_i(x) \geq t \mathit{l}^2 \alpha_0^h >0, \ i=0,1,2, \quad
	b_j(x) \geq t^3 \mathit{l}^2 \beta_0^h >0, \ j=0,1, \quad \hbox{a.e. in } \ \Omega,
\end{equation}
where $\alpha_0^h = \frac{2}{15}\alpha_0$ and $\beta_0^h = \frac{1}{30} \alpha_0$.

By \eqref{eq:8-1}, \eqref{eq:8-2} we obtain the following strong convexity conditions on $\mathbb{P}+\mathbb{P}^h$ and $\mathbb{Q}$. For every $A \in \widehat{\M}^2$ we have 
\begin{equation}
	\label{eq:8-3}
	(\mathbb{P}+\mathbb{P}^h) A \cdot A \geq t(t^2+l^2) 
	\xi_{\mathbb{P}} |A|^2 \quad \hbox{a.e. in } \ \Omega;
\end{equation}
for every $B \in \widehat{\M}^3$ we have 
\begin{equation}
	\label{eq:8-4}
	\mathbb{Q} B \cdot B \geq t^3l^2 
	\xi_{\mathbb{Q}} |B|^2 \quad \hbox{a.e. in } \ \Omega;
\end{equation}
where $\xi_{\mathbb{P}}$, $\xi_{\mathbb{Q}}$ are positive constants only depending on $\alpha_0$ and $\gamma_0$. 

iii) \textit{Bounds on the jumps and uniform strong convexity for $\mathbb{P}+\mathbb{P}^h$ and $\mathbb{Q}$.}

Either there exists $\eta >0$, $\overline{\eta} >0$ and $\delta >1$, $\overline{\delta}>1$ such that
\begin{equation}
	\label{eq:9-1}
	\eta (\mathbb{P}+\mathbb{P}^h) \leq
	 (\widetilde{\mathbb{P}}+\widetilde{\mathbb{P}}^h) -  (\mathbb{P}+\mathbb{P}^h)
	 \leq (\delta -1)  (\mathbb{P}+\mathbb{P}^h)  \quad \hbox{a.e. in } \ \Omega,
\end{equation}
\begin{equation}
	\label{eq:9-2}
	\overline{\eta} \mathbb{Q} \leq
	\widetilde{\mathbb{Q}} -  \mathbb{Q}
	\leq (\overline{\delta} -1)  \mathbb{Q}  \quad \hbox{a.e. in } \ \Omega,
\end{equation}
or there exists $\eta >0$, $\overline{\eta} >0$ and $0<\delta <1$, $0<\overline{\delta}<1$ such that
\begin{equation}
	\label{eq:9-3}
  \eta(\mathbb{P}+\mathbb{P}^h) \leq
	(\mathbb{P}+\mathbb{P}^h) - (\widetilde{\mathbb{P}}+\widetilde{\mathbb{P}}^h)
	\leq 	(1-\delta) (\mathbb{P}+\mathbb{P}^h)  \quad \hbox{a.e. in } \ \Omega,
\end{equation}
\begin{equation}
	\label{eq:9-4}
	\overline{\eta} \mathbb{Q} \leq
	\mathbb{Q} - \widetilde{\mathbb{Q}}
	\leq (1-\overline{\delta}) \mathbb{Q}  \quad \hbox{a.e. in } \ \Omega.
\end{equation}
Let us note that assumptions ii) and iii) ensure that $\widetilde{\mathbb{P}}+\widetilde{\mathbb{P}}^h$ and $\widetilde{\mathbb{Q}}$ are strongly convex a.e. in $\Omega$.

iv) \textit{Regularity for $\mathbb{P}$, $\mathbb{P}^h$ and $\mathbb{Q}$.}

We assume $\mathbb{P}$, $\mathbb{P}^h \in C^{1,1}(\overline{\Omega})$ and $\mathbb{Q} \in C^{2,1}(\overline{\Omega})$, with
\begin{equation}
	\label{eq:9-5}
	\|\mathbb{P}\|_{ C^{1,1}(\overline{\Omega})  }
	+
	\|\mathbb{P}^h\|_{ C^{1,1}(\overline{\Omega})  }
	+r_0^{-2}
	\|\mathbb{Q}\|_{ C^{2,1}(\overline{\Omega})  }
	\leq M_2r_0^3,
\end{equation}
with $M_2$ depending on $\frac{t}{r_0}$, $\frac{l}{r_0}$.

On the boundary data appearing in \eqref{eq:1-3}--\eqref{eq:1-5} we assume
\begin{equation}
	\label{eq:10-1}
	\widehat{V} \in H^{ - 3/2  }(\partial \Omega), \quad 
	\widehat{M}_n \in H^{ - 1/2  }(\partial \Omega), \quad 
	\widehat{M}_n^h \in H^{  1/2  }(\partial \Omega)
\end{equation}
and we obviously assume the compatibility conditions \eqref{eq:4-2}.

In what follows we denote by $u$, $u_0$ the solutions of the equilibrium problem for the nanoplate \eqref{eq:1-2}--\eqref{eq:1-5} with and without inclusion, namely $u\in H^3(\Omega)$ is the solution to \eqref{eq:1-2}--\eqref{eq:1-5} when  $M(u)=-(\chi_{\Omega \setminus D} (\mathbb{P}+\mathbb{P}^h   ) + \chi_D ( \widetilde{\mathbb{P}} + \widetilde{\mathbb{P}}^h   ))D^2u$, $\overline{M}^h(u)=(\chi_{\Omega \setminus D} \mathbb{Q} + \chi_D \widetilde{\mathbb{Q}})D^3u$ and $u_0\in H^3(\Omega)$ is the solution to \eqref{eq:1-2}--\eqref{eq:1-5} when $M(u_0)=-(\mathbb{P}+\mathbb{P}^h)D^2u_0$, $\overline{M}^h(u_0)=\mathbb{Q}D^3u_0$. Let us recall that $u$ and $u_0$ are uniquely determined by the normalization conditions \eqref{eq:4-3}.

Note that the boundary data $\widehat{V}$, $\widehat{M}_n$, $\widehat{M}_n^h$ associated to the problem for $u$ and $u_0$ are the \textit{same}. 

Finally, let us introduce the quantities
\begin{equation}
	\label{eq:11-1}
	W=L(u) = -\int_{\partial \Omega}  \widehat{V} u + \widehat{M}_n u,_{n} + \widehat{M}^h_n u,_{nn},
\end{equation}
\begin{equation}
	\label{eq:11-2}
	W_0=L(u_0) = -\int_{\partial \Omega}  \widehat{V} u_0 + \widehat{M}_n u_{0,n} + \widehat{M}^h_n u_{0,nn},
\end{equation}
which represent the works exerted by the boundary data when the inclusion $D$ is present or absent, respectively. By the weak formulation of the corresponding problems, the works $W$ and $W_0$ coincide with the strain energy stored in the deformed microplate, namely
\begin{multline}
	\label{eq:11-3}
	W= 
	\int_\Omega
	(\chi_{\Omega \setminus D} (\mathbb{P}+\mathbb{P}^h   ) + \chi_D ( \widetilde{\mathbb{P}} + \widetilde{\mathbb{P}}^h   )) D^2 u \cdot D^2 u+
	\\
	+ 
	(\chi_{\Omega \setminus D} \mathbb{Q} + \chi_D \widetilde{\mathbb{Q}})
	D^3 u \cdot D^3 u,
\end{multline}
\begin{equation}
	\label{eq:11-4}
	W_0= 
	\int_\Omega
	(\mathbb{P}+\mathbb{P}^h   ) D^2 u_0 \cdot D^2 u_0+
	\mathbb{Q}	D^3 u_0 \cdot D^3 u_0.
\end{equation}

\subsection{Main results}
\label{sec:main-results}

We are now in position to state our size estimates results for nanoplates.
\begin{theo} [{\bf Lower bound of $|D|$}]
	\label{theo:13-1}
	Let $\Omega$ be a bounded domain in $\R^2$ such that $\partial \Omega$ is of $C^{2,1}$-class with constants $r_0$, $M_0$ and satisfying \eqref{eq:1-1}. Let $D$, $D \subset \subset \Omega$, be a measurable subset of $\Omega$ satisfying 
	\begin{equation}
		\label{eq:13-1}
		dist(D, \partial \Omega) \geq d_0 r_0,
	\end{equation}
	where $d_0$ is a positive constant. Let the tensors $\mathbb{P}$, $\mathbb{P}^h$, $\widetilde{\mathbb{P}}$,  $\widetilde{\mathbb{P}}^h \in L^\infty( \Omega, \mathcal{L}(\widehat{\M}^2, \widehat{\M}^2   )  )$ and  $\mathbb{Q}$, $\widetilde{\mathbb{Q}} \in L^\infty( \Omega, \mathcal{L}(\widehat{\M}^3, \widehat{\M}^3   )  )$ satisfy the symmetry conditions \eqref{eq:3-1}, \eqref{eq:3-2} and \eqref{eq:3-6}, the strong convexity conditions \eqref{eq:3-3} and \eqref{eq:3-7}, and either the jump conditions \eqref{eq:9-1}--\eqref{eq:9-2} or \eqref{eq:9-3}--\eqref{eq:9-4}. Moreover, let the tensors $\mathbb{P}$, $\mathbb{P}^h$,  $\mathbb{Q}$ satisfy the regularity conditions iv). 
	
	If 	\eqref{eq:9-1}--\eqref{eq:9-2} hold, then we have
	\begin{equation}
		\label{eq:13-2}
		|D| \geq C_1^+ r_0^2 \frac{W_0-W}{W}.
	\end{equation}
	If, conversely, \eqref{eq:9-3}--\eqref{eq:9-4} hold, then we have
	\begin{equation}
		\label{eq:13-3}
		|D| \geq C_1^- r_0^2 \frac{W-W_0}{W_0}.
	\end{equation}
	Here the constants $C_1^+$, $C_1^-$ depend only on $\frac{t}{r_0}$, $M_0$, $M_1$, $d_0$, $\xi_{\mathbb P}$, $\xi_{\mathbb Q}$, $M_2$, $\delta$, $\overline{\delta}$.
	
	%
\end{theo}
\begin{theo} [{\bf Upper bound of $|D|$ for fat inclusions}]
	\label{theo:14-1}
	Let $\Omega$ be a bounded domain in $\R^2$ such that $\partial \Omega$ is of $C^{3,1}$-class with constants $r_0$, $M_0$ and satisfying \eqref{eq:1-1}. Let $D$ be a measurable subset of $\Omega$ satisfying 
	\begin{equation}
		\label{eq:14-1}
		|D_{h_1r_0}| \geq \frac{1}{2} |D|,
	\end{equation}
	for a given positive constant $h_1$. Let the tensors $\mathbb{P}$, $\mathbb{P}^h \in C^{1,1}( \overline{\Omega}, \mathcal{L}(\widehat{\M}^2, \widehat{\M}^2   )  )$ and $\mathbb{Q} \in C^{2,1}( \overline{\Omega}, \mathcal{L}(\widehat{\M}^3, \widehat{\M}^3   )  )$ satisfy the isotropy conditions \eqref{eq:6-1}, \eqref{eq:6-2} and \eqref{eq:7-0}, respectively, and let the Lamé moduli $\mu$ and $\lambda$ satisfy the strong convexity conditions \eqref{eq:8-1}. Let the tensors 
	 $\widetilde{\mathbb{P}}$,  $\widetilde{\mathbb{P}}^h \in L^\infty( \Omega, \mathcal{L}(\widehat{\M}^2, \widehat{\M}^2   )  )$ and  $\widetilde{\mathbb{Q}} \in L^\infty( \Omega, \mathcal{L}(\widehat{\M}^3, \widehat{\M}^3   )  )$ and let us assume the jump conditions iii).
	
	If 	\eqref{eq:9-1}--\eqref{eq:9-2} hold, then we have
	\begin{equation}
		\label{eq:14-2}
		|D| \leq C_2^+ r_0^2 \frac{W_0-W}{W_0}.
	\end{equation}
	If, conversely, \eqref{eq:9-3}--\eqref{eq:9-4} hold, then we have
	\begin{equation}
		\label{eq:14-3}
		|D| \leq C_2^- r_0^2 \frac{W-W_0}{W_0}.
	\end{equation}
	Here the constants $C_2^+$, $C_2^-$ depend only on $\frac{t}{r_0}$, $\frac{l}{r_0}$, $M_0$, $M_1$, $h_1$, $\alpha_0$, $\gamma_0$, $M_2$, $\eta$, $\overline{\eta}$, $\delta$, $\overline{\delta}$ and on the ratio
	\begin{equation}
		\label{eq:14-4}
		F = 
		\frac
		{\|\widehat{V}\|_{H^{ - 3/2  }(\partial \Omega)} + r_0^{-1}\| \widehat{M}_n\|_{H^{ - 1/2  }(\partial \Omega)} + r_0^{-2}\|  \widehat{M}_n^h\|_{H^{  1/2  }(\partial \Omega)}}
		{\|\widehat{V}\|_{H^{ - 5/2  }(\partial \Omega)} + r_0^{-1}\| \widehat{M}_n\|_{H^{ - 3/2  }(\partial \Omega)} + r_0^{-2}\|  \widehat{M}_n^h\|_{H^{ - 1/2  }(\partial \Omega)}}.
	\end{equation}
\end{theo}
\begin{theo} [{\bf Upper bound of $|D|$ for general inclusions}]
	\label{theo:15-1}
	Let $\Omega$ be a bounded domain in $\R^2$ such that $\partial \Omega$ is of $C^{3,1}$-class with constants $r_0$, $M_0$ and satisfying \eqref{eq:1-1}. Let $D$, $D \subset \subset \Omega$, be a measurable subset of $\Omega$   satisfying 
	\begin{equation}
		\label{eq:13-1bis}
		dist(D, \partial \Omega) \geq d_0 r_0,
	\end{equation}
	where $d_0$ is a positive constant.
	Let the tensors $\mathbb{P}$, $\mathbb{P}^h \in C^{1,1}( \overline{\Omega}, \mathcal{L}(\widehat{\M}^2, \widehat{\M}^2   )  )$ and $\mathbb{Q} \in C^{2,1}( \overline{\Omega}, \mathcal{L}(\widehat{\M}^3, \widehat{\M}^3   )  )$ satisfy the isotropy conditions \eqref{eq:6-1}, \eqref{eq:6-2} and \eqref{eq:7-0}, respectively, and let the Lamé moduli $\mu$ and $\lambda$ satisfy the strong convexity conditions \eqref{eq:8-1}. Let the tensors 
	$\widetilde{\mathbb{P}}$,  $\widetilde{\mathbb{P}}^h \in L^\infty( \Omega, \mathcal{L}(\widehat{\M}^2, \widehat{\M}^2   )  )$ and  $\widetilde{\mathbb{Q}} \in L^\infty( \Omega, \mathcal{L}(\widehat{\M}^3, \widehat{\M}^3   )  )$ and let us assume the jump conditions iii).
	
	If 	\eqref{eq:9-1}--\eqref{eq:9-2} hold, then we have
	\begin{equation}
		\label{eq:15-1}
		|D| \leq C_2^+ r_0^2 \left ( \frac{W_0-W}{W_0}  \right )^{1/p} .
	\end{equation}
	If, conversely, \eqref{eq:9-3}--\eqref{eq:9-4} hold, then we have
	\begin{equation}
		\label{eq:15-2}
		|D| \leq C_2^- r_0^2 \left ( \frac{W-W_0}{W_0}  \right )^{1/p}.
	\end{equation}
	Here the constants $C_2^+$, $C_2^-$ and $p>1$ depend only on $\frac{t}{r_0}$, $\frac{l}{r_0}$, $M_0$, $M_1$, $d_0$, $\alpha_0$, $\gamma_0$, $M_2$, $\eta$, $\overline{\eta}$, $\delta$, $\overline{\delta}$ and on the ratio $F$ given in \eqref{eq:14-4}.
\end{theo}

\section{Proof of Theorem \ref{theo:13-1}}
\label{sec:Proof-lower-bound-di-D}

Let us premise the following Energy lemma, which states that the work gap $|W-W_0|$ is estimated {}from above and {}from below by the strain energy of the unperturbed referential solution $u_0$ stored in the inclusion $D$.

\begin{lem}[\textbf{Energy lemma}]
	\label{lem:16.1}
	Let $\Omega$ be a bounded domain in $\R^2$, such that $\partial \Omega$ is of $C^{2,1}$-class. Let $D$ be a measurable subset of $\Omega$. Let the tensors $\mathbb{P}$, $\mathbb{P}^h$, $\widetilde{\mathbb{P}}$,  $\widetilde{\mathbb{P}}^h \in L^\infty( \Omega, \mathcal{L}(\widehat{\M}^2, \widehat{\M}^2   )  )$ and $\mathbb{Q}$, $\widetilde{\mathbb{Q}} \in L^\infty( \Omega, \mathcal{L}(\widehat{\M}^3, \widehat{\M}^3   )  )$ satisfy the symmetry conditions \eqref{eq:3-1}, \eqref{eq:3-2}, \eqref{eq:3-6}, respectively. Let $\xi_0$, $\xi_1$, $\overline{\xi_0}$, $\overline{\xi_1}$, $  0 < \xi_0 < \xi_1$, $0 < \overline{\xi_0} < \overline{\xi_1}$, be such that
	\begin{equation}
		\label{eq:16-1}
		t^3 \xi_0 |A|^2 \leq ( \mathbb{P}(x) +\mathbb{P}^h(x) ) A \cdot A \leq t^3 \xi_1 |A|^2 \quad \hbox{for a.e. } x \in \Omega,
	\end{equation}
	\begin{equation}
		\label{eq:16-2}
		t^5 \overline{\xi_0} |B|^2 \leq \mathbb{Q}(x) B \cdot B \leq t^5 \overline{\xi_1} |B|^2 \quad \hbox{for a.e. } x \in \Omega,
	\end{equation}
for every matrix $A \in \widehat{\M}^2$ and $B \in \widehat{\M}^3$. Let the jumps $ (  \widetilde{\mathbb{P}}+\widetilde{\mathbb{P}}^h   )-( \mathbb{P}+\mathbb{P}^h   )$, $\widetilde{\mathbb{Q}} -  \mathbb{Q}$ satisfy either \eqref{eq:9-1}--\eqref{eq:9-2} or \eqref{eq:9-3}--\eqref{eq:9-4}. Let $u, u_0 \in H^3(\Omega)$ be the weak solutions to the problem \eqref{eq:1-2}--\eqref{eq:1-5}, normalized by \eqref{eq:4-3}, when the inclusion $D$ is present or absent, respectively, for the Neumann data $\widehat{V} \in H^{ - 5/2  }(\partial \Omega)$, $\widehat{M}_n \in H^{ - 3/2  }(\partial \Omega)$, $\widehat{M}_n^h \in H^{  -1/2  }(\partial \Omega)$ that fulfill the compatibility conditions \eqref{eq:4-2}.

If \eqref{eq:9-1}--\eqref{eq:9-2} hold, then
\begin{equation}
	\label{eq:17-1}
	\frac{\eta_* \xi_{0*} t^3}{\delta^*} \int_D |D^2 u_0|^2 +t^2 |D^3 u_0|^2
	\leq W_0 -W 
	\leq (\delta^*-1)\xi_1^* t^3 \int_D |D^2 u_0|^2 +t^2 |D^3 u_0|^2;
\end{equation}

if, conversely, \eqref{eq:9-3}--\eqref{eq:9-4} hold, then
\begin{equation}
	\label{eq:17-2}
	\eta_*\xi_{0*} t^3 \int_D |D^2 u_0|^2 +t^2 |D^3 u_0|^2
	\leq W -W_0 
	\leq \frac{(1-\delta_*)\xi_1^*t^3}{\delta_*} \int_D |D^2 u_0|^2 +t^2 |D^3 u_0|^2.
\end{equation}
Here $\eta_*=\min\{\eta, \overline{\eta} \}$, $\delta^*=\max\{\delta, \overline{\delta} \}$, $\xi_{0*}=\min\{\xi_0, \overline{\xi}_0 \}$, $\xi_1^*=\max\{\xi_1, \overline{\xi}_1 \}$, $\delta_*=\min\{\delta, \overline{\delta} \}$.
\end{lem}

\begin{rem}
Let us note that if the materials constituting the inclusion $D$ and the surrounding material in $\Omega \setminus D$ are isotropic with Lamé moduli $\widetilde{\mu}$, $\widetilde{\lambda}$ and $\mu$, $\lambda$, respectively, then the jump conditions \eqref{eq:9-1}, \eqref{eq:9-2} and \eqref{eq:9-3}, \eqref{eq:9-4} can be written in terms of the difference $\widetilde{\mu} - \mu$ and $\widetilde{\kappa} - \kappa$, where $\widetilde{\kappa} = \frac{2\widetilde{\mu}(2 \widetilde{\mu} +3\widetilde{\lambda})  }{2 \widetilde{\mu} +\widetilde{\lambda}}$, ${\kappa} = \frac{2{\mu}(2 {\mu} +3{\lambda})  }{2 {\mu} +{\lambda}}$.
\end{rem}
	
\begin{proof}
Let us assume conditions \eqref{eq:9-1}--\eqref{eq:9-2} (i.e. the material of the inclusion $D$ is stiffer than the surrounding material in $\Omega \setminus D$) and prove inequalities \eqref{eq:17-1}. The proof of the estimates \eqref{eq:17-2} is similar.

We start by determining some basic identities. Let us denote by $u_1 \in H^3(\Omega)$ and $u_2\in H^3(\Omega)$ the weak solution to \eqref{eq:1-2}--\eqref{eq:1-5} when the inclusion $D_1$ or $D_2$ is present, respectively. Since the boundary data are the same, for every $w \in H^3(\Omega)$ we have
\begin{multline}
	\label{eq:LE3-3}
	\int_\Omega 
	((\mathbb{P}+\mathbb{P}^h)+\chi_{D_1}\mathbb{H}_{\mathbb{P}}) 
	D^2 u_1 \cdot D^2 w +
	(\mathbb{Q}+\chi_{D_1}\mathbb{H}_{\mathbb{Q}}) 
	D^3 u_1 \cdot D^3 w
	=
	\\
	=
	\int_\Omega 
	((\mathbb{P}+\mathbb{P}^h)+\chi_{D_2}\mathbb{H}_{\mathbb{P}}) 
	D^2 u_2 \cdot D^2 w +
	(\mathbb{Q}+\chi_{D_2}\mathbb{H}_{\mathbb{Q}}) 
	D^3 u_2 \cdot D^3 w,
\end{multline}
where we have defined
\begin{equation}
	\label{eq:def-salto}
	\mathbb{H}_{\mathbb{P}} = (\widetilde{\mathbb{P}}+\widetilde{\mathbb{P}}^h) -  (\mathbb{P}+\mathbb{P}^h), \quad 
	\mathbb{H}_{\mathbb{Q}} = \widetilde{\mathbb{Q}}
	-
	{\mathbb{Q}}.
\end{equation}
Note that the tensors $\mathbb{H}_{\mathbb{P}}$, $\mathbb{H}_{\mathbb{Q}}$ satisfy the symmetry conditions  \eqref{eq:3-1}--\eqref{eq:3-2}) and \eqref{eq:3-6}, respectively.

By subtracting $\int_\Omega 
((\mathbb{P}+\mathbb{P}^h)+\chi_{D_1}\mathbb{H}_{\mathbb{P}}) 
D^2 u_2 \cdot D^2 w +
(\mathbb{Q}+\chi_{D_1}\mathbb{H}_{\mathbb{Q}}) 
D^3 u_2 \cdot D^3 w $ to both sides of \eqref{eq:LE3-3} we obtain
\begin{multline}
	\label{eq:LE4-1}
	\int_\Omega 
	((\mathbb{P}+\mathbb{P}^h)+\chi_{D_1}\mathbb{H}_{\mathbb{P}}) 
	D^2 (u_1-u_2) \cdot D^2 w +
	(\mathbb{Q}+\chi_{D_1}\mathbb{H}_{\mathbb{Q}}) 
	D^3 (u_1-u_2) \cdot D^3 w
	=
	\\
	=
	\int_\Omega 
	(\chi_{D_2}-\chi_{D_1})
	(\mathbb{H}_{\mathbb{P}}
	D^2 u_2 \cdot D^2 w +
	\mathbb{H}_{\mathbb{Q}}
	D^3 u_2 \cdot D^3 w),
\end{multline}
for every $w \in H^3(\Omega)$.

Let us choose $w=u_1$ in \eqref{eq:LE4-1}. By using the symmetry conditions on $\mathbb{P}$, $\mathbb{P}^h$, $\mathbb{Q}$, $\mathbb{H}_{\mathbb{P}}$ and $\mathbb{H}_{\mathbb{Q}}$, and considering $u_1-u_2$ as test function in the weak formulation for $u_1$, we have
\begin{multline}
	\label{eq:LE5-1}
	-\int_{\partial \Omega}  \widehat{V} (u_1-u_2) + \widehat{M}_n (u_{1,n} - u_{2,n}) + \widehat{M}^h_n (u_{1,nn} - u_{2,nn})
	=
	\\
	=
	\int_\Omega 
	(\chi_{D_2}-\chi_{D_1})
	(\mathbb{H}_{\mathbb{P}}
	D^2 u_2 \cdot D^2 u_1 +
	\mathbb{H}_{\mathbb{Q}}
	D^3 u_2 \cdot D^3 u_1).
\end{multline}
Next, we choose $w=u_1-u_2$ in \eqref{eq:LE4-1}. By using \eqref{eq:LE5-1}, after simple algebra we obtain the following identity
\begin{multline}
	\label{eq:LE6-1}
	\int_\Omega 
	((\mathbb{P}+\mathbb{P}^h)+\chi_{D_1}\mathbb{H}_{\mathbb{P}}) 
	D^2 (u_1-u_2) \cdot D^2 (u_1-u_2) +
	(\mathbb{Q}+\chi_{D_1}\mathbb{H}_{\mathbb{Q}}) 
	D^3 (u_1-u_2) \cdot D^3 (u_1-u_2)
	+
	\\
	+
	\int_{D_2 \setminus D_1} 
	\mathbb{H}_{\mathbb{P}}
	D^2 u_2 \cdot D^2 u_2 +
	\mathbb{H}_{\mathbb{Q}}
	D^3 u_2 \cdot D^3 u_2
	=
	\\
	=
	-\int_{\partial \Omega}  \widehat{V} (u_1-u_2) + \widehat{M}_n (u_{1,n} - u_{2,n}) + \widehat{M}^h_n (u_{1,nn} - u_{2,nn})
	+
	\\
	+
	\int_{D_1 \setminus D_2} 
	\mathbb{H}_{\mathbb{P}}
	D^2 u_2 \cdot D^2 u_2 +
	\mathbb{H}_{\mathbb{Q}}
	D^3 u_2 \cdot D^3 u_2.
\end{multline}
By choosing $D_1=D$ (i.e. $u_1=u$) and $D_2=\emptyset$ (i.e. $u_2=u_0$) in \eqref{eq:LE6-1}, we obtain the first fundamental identity
\begin{multline}
	\label{eq:LE6-2}
	\int_\Omega 
	((\mathbb{P}+\mathbb{P}^h)\chi_{\Omega \setminus D} +
	(\widetilde{\mathbb{P}}+ \widetilde{\mathbb{P}}^h  )\chi_{D})
	D^2 (u-u_0) \cdot D^2 (u-u_0) +
	\\
	+ \int_\Omega 
	(\mathbb{Q}\chi_{\Omega \setminus D}+
	\widetilde{\mathbb{Q}}\chi_D) 
	D^3 (u-u_0) \cdot D^3 (u-u_0)
	-
	\\
	-
	\int_{D} 
	\mathbb{H}_{\mathbb{P}}
	D^2 u_0 \cdot D^2 u_0 +
	\mathbb{H}_{\mathbb{Q}}
	D^3 u_0 \cdot D^3 u_0=
	\\
	=
	-\int_{\partial \Omega}  \widehat{V} (u-u_0) + \widehat{M}_n (u_{,n} - u_{0,n}) + \widehat{M}^h_n (u_{,nn} - u_{0,nn})=
	\\
	=W-W_0.
\end{multline}
A second fundamental identity is obtained by choosing $D_1= \emptyset$ ($u_1=u_0$) and $D_2=D$ ($u_2=u$) in \eqref{eq:LE6-1}:
\begin{multline}
	\label{eq:LE7-1}
	\int_\Omega 
	(\mathbb{P}+\mathbb{P}^h) 
	D^2 (u_0-u) \cdot D^2 (u_0-u) +
	\mathbb{Q}	D^3 (u_0-u) \cdot D^3 (u_0-u)
	+
	\\
	+
	\int_{D} 
	\mathbb{H}_{\mathbb{P}}
	D^2 u \cdot D^2 u +
	\mathbb{H}_{\mathbb{Q}}
	D^3 u \cdot D^3 u
	=
	\\
	=
	-\int_{\partial \Omega}  \widehat{V} (u_0-u) + \widehat{M}_n (u_{0,n} - u_{,n}) + \widehat{M}^h_n (u_{0,nn} - u_{,nn})=
	\\
	=W_0-W.
\end{multline}
Next, let us choose $w=u_0$ as test function in the weak formulation \eqref{eq:4bis-1} of the Neumann problem \eqref{eq:1-2}--\eqref{eq:1-5} when the inclusion $D$ is present, obtaining
\begin{multline}
	\label{eq:LE7-2}
	\int_\Omega 
	((\mathbb{P}+\mathbb{P}^h) +
	\chi_D \mathbb{H}_{ \mathbb{P} })
	D^2 u \cdot D^2 u_0 +
	(\mathbb{Q}	+ \chi_D  \mathbb{H}_{ \mathbb{Q} }) D^3 u \cdot D^3 u_0 =
	\\
	=
	-\int_{\partial \Omega}  \widehat{V} u_0 + \widehat{M}_n u_{0,n} + \widehat{M}^h_n u_{0,nn}.
\end{multline}
Conversely, choosing the solution $u$ of \eqref{eq:1-2}--\eqref{eq:1-5} when the inclusion $D$ is present as test function in the weak formulation \eqref{eq:4bis-1} when the inclusion is absent, we have
\begin{equation}
	\label{eq:LE7-3}
	\int_\Omega 
	(\mathbb{P}+\mathbb{P}^h)
	D^2 u_0 \cdot D^2 u +
	\mathbb{Q} D^3 u_0 \cdot D^3 u 
	=
	-\int_{\partial \Omega}  \widehat{V} u + \widehat{M}_n u_{,n} + \widehat{M}^h_n u_{,nn}.
\end{equation}
By subtracting \eqref{eq:LE7-3} {}from \eqref{eq:LE7-2}, we obtain a third fundamental identity
\begin{multline}
	\label{eq:LE7-4}
	\int_D 
	\mathbb{H}_{ \mathbb{P} }
	D^2 u \cdot D^2 u_0 +
	\mathbb{H}_{ \mathbb{Q} } D^3 u \cdot D^3 u_0 =
	\\
	=
	\int_{\partial \Omega}  \widehat{V} (u_0-u) + \widehat{M}_n (u_{0,n} - u_{,n})+ \widehat{M}^h_n (u_{0,nn}-u_{,nn}).
\end{multline}
We are now in position to derive the estimates \eqref{eq:17-1}.

Using the positivity of $\mathbb{P}+\mathbb{P}^h$, $\widetilde{\mathbb{P}}+ \widetilde{\mathbb{P}}^h$, $\mathbb{Q}$ and $\widetilde{\mathbb{Q}}$, {}from the first identity \eqref{eq:LE6-2} we obtain 
\begin{multline}
	\label{eq:LE8-1}
	W_0 - W = -\int_{\partial \Omega}  \widehat{V} (u_0-u) + \widehat{M}_n (u_{0,n}-u_{,n}) + \widehat{M}^h_n (u_{0,nn} - u_{,nn})
	\leq
	\\
	\leq
	\int_{D} 
	\mathbb{H}_{\mathbb{P}}
	D^2 u_0 \cdot D^2 u_0 +
	\mathbb{H}_{\mathbb{Q}}
	D^3 u_0 \cdot D^3 u_0
\end{multline}
and the estimate from above of the work gap $W_0-W$ in \eqref{eq:17-1} easily follows {}from \eqref{eq:9-1}--\eqref{eq:9-2} and {}from \eqref{eq:16-1}--\eqref{eq:16-2}.

To get the estimate {}from below of $W_0-W$, we use the following inequality
\begin{multline}
	\label{eq:LE9-1}
	\int_{D} 
	\mathbb{H}_{\mathbb{P}}
	D^2 u_0 \cdot D^2 u_0 +
	\mathbb{H}_{\mathbb{Q}}
	D^3 u_0 \cdot D^3 u_0
	\leq
	\\
	\leq
	(1+\epsilon) \int_{D}
	\mathbb{H}_{\mathbb{P}} D^2 (u-u_0) \cdot D^2 (u-u_0)
	+
	\left (  1+ \frac{1}{\epsilon}  \right )
	\int_{D}
	\mathbb{H}_{\mathbb{P}} D^2 u \cdot D^2 u
	+
	\\
	+
	(1+\overline{\epsilon}) \int_{D}
	\mathbb{H}_{\mathbb{Q}} D^3 (u-u_0) \cdot D^3 (u-u_0)
	+
	\left (  1+ \frac{1}{\overline{\epsilon}}  \right )
	\int_{D}
	\mathbb{H}_{\mathbb{Q}} D^3 u \cdot D^3 u,
\end{multline}
for every $\epsilon >0$, $\overline{\epsilon} >0$. The above inequality follows by the symmetry properties \eqref{eq:3-1}, \eqref{eq:3-2}, \eqref{eq:3-6}  and the positivity conditions \eqref{eq:9-1}, \eqref{eq:9-2}.

By using the jump conditions \eqref{eq:9-1}--\eqref{eq:9-2}, by choosing $\epsilon= (\delta -1)^{-1}$, $\overline{\epsilon}=	\left ( \overline{\delta} -1 \right )^{-1}$ in \eqref{eq:LE9-1}, and employing identity \eqref{eq:LE7-1} we get
\begin{multline}
	\label{eq:lower_bound_work_gap}
	\int_{D} 
	\mathbb{H}_{\mathbb{P}}
	D^2 u_0 \cdot D^2 u_0 +
	\mathbb{H}_{\mathbb{Q}}
	D^3 u_0 \cdot D^3 u_0
	\leq
	\\
	\leq
	(1+\epsilon) \int_{D}
	(\delta -1)( \mathbb{P} + \mathbb{P}^h)  D^2 (u-u_0) \cdot D^2 (u-u_0)
	+
	\left (  1+ \frac{1}{\epsilon}  \right )
	\int_{D}
	\mathbb{H}_{\mathbb{P}} D^2 u \cdot D^2 u
	+
	\\
	+
	(1+\overline{\epsilon}) \int_{D}
	(\overline{\delta}-1)
	\mathbb{Q} D^3 (u-u_0) \cdot D^3 (u-u_0)
	+
	\left (  1+ \frac{1}{\overline{\epsilon}}  \right )
	\int_{D}
	\mathbb{H}_{\mathbb{Q}} D^3 u \cdot D^3 u
	\leq
	\\
	\leq
	\max\{\delta, \overline{\delta}\}
	\left \{ 
	\int_\Omega 
	(\mathbb{P}+\mathbb{P}^h) 
	D^2 (u_0-u) \cdot D^2 (u_0-u) +
	\mathbb{Q}	D^3 (u_0-u) \cdot D^3 (u_0-u)
	\right.
	+
	\\
	+
	\left .
	\int_{D} 
	\mathbb{H}_{\mathbb{P}}
	D^2 u \cdot D^2 u +
	\mathbb{H}_{\mathbb{Q}}
	D^3 u \cdot D^3 u
	\right \}
	= \max\{\delta, \overline{\delta}\}(W_0 -W).
\end{multline}
Finally, the estimate from below of $W_0-W$ in \eqref{eq:17-1} follows from \eqref{eq:lower_bound_work_gap} and \eqref{eq:9-1}--\eqref{eq:9-2}.
\end{proof}	

\begin{proof}[Proof of Theorem \ref{theo:13-1}]
To fix the ideas, let us assume that the jump conditions \eqref{eq:9-1}--\eqref{eq:9-2} hold. Let us estimate the right hand side of \eqref{eq:17-1}. Let us notice that there exists $d^{*}, 0<d^{*}<d_0$, only depending on $M_0$, such that $\Omega_{d^*r_0}$ is of Lipschitz class with constants $\gamma r_0$, $\gamma'M_0$, where $\gamma<1$ and $\gamma'>1$ only depend on $M_0$, and $D\subset \Omega_{d^*r_0}$(see \cite[Lemma 14.16]{G-Tr} for details). By Sobolev Imbedding Theorem (\cite[Chap.5, Theorem 5.4]{Adams}), interior regularity estimates \eqref{eq:5ter-3}, standard Poincar\`e inequality (see \cite[Proposition 3.3]{MRV2007} ), by \eqref{eq:9-5} and \eqref{eq:13-1}, we obtain
 \begin{eqnarray}
 &&W_0-W \le C r_0^3\int_{D} |{D}^2 u_0|^2 + r_0^2|{D}^3 u_0|^2 \le C r_0^3 |D| \left(\|{D}^2 u\|^2_{L^{\infty}(D)} + r_0^2 \| {D}^3 u\|^2_{L^{\infty}(D)} \right)\le\nonumber  \\
 &&C \frac{1}{r_0} |D| \|u_0\|^2_{H^6({\Omega}_{d^*r_0})}\le C \frac{1}{r_0} |D| \|u_0\|^2_{H^3({\Omega})}\le C r_0 |D| \int_{\Omega} |D^2 u_0|^2 + r_0^2|D^3 u_0|^2\le \nonumber \\
&& \frac{|D|}{r_0^2}\int_{\Omega} (\mathbb{P}+\mathbb{P}^h)D^2 u_0\cdot D^2 u_0 + \mathbb{Q}D^3 u_0\cdot D^3 u_0 = \frac{C}{r_0^2} |D| W_0 \ , 
 \end{eqnarray}
where $C>0$ depends on $d_0, \delta, \bar{\delta}, \frac{t}{r_0}, M_0, M_1, \xi_{\mathbb{P}}, \xi_{\mathbb{Q}}, M_2.$ Hence, estimate \eqref{eq:13-2} follows.

\end{proof}

\section{Doubling and three spheres inequality for the Hessian}	\label{sec:DIandTS}
This section is devoted to strong unique continuation estimates for solutions to equation \eqref{eq:1-2} for the isotropic case only. Such estimates are given in the form of  doubling inequality and three spheres inequality for the hessian of the solutions. The latter are crucial tools of unique continuation needed in the proof of upper bound estimates for the size of both the considered inclusions. We shall premise the proof of the main result of this section with some auxiliary results which are contained in \cite{MRSV2022}.

For simplicity of notation in this Section we denote by $u$ a weak solution to the partial differential equation \eqref{eq:1-2}.

\begin{prop} [{\bf Doubling inequality and three sphere inequality for solutions to \eqref{eq:1-2}}]
\label{prop:SUCP-1}
Let $\mathbb{P}, \mathbb{P}^h \in C^{1,1}(\overline{B_1}, {\cal L} (\widehat{\mathbb{M}}^2, \widehat{\mathbb{M}}^2)), \mathbb{Q} \in C^{2,1}(\overline{B_1}, {\cal L} (\widehat{\mathbb{M}}^3,\widehat{\mathbb{M}}^3))$ be given by \eqref{eq:6-1}, \eqref{eq:6-2}, \eqref{eq:7-0} and satisfying the regularity condition \eqref{eq:9-5}, the strong convexity conditions \eqref{eq:8-3}, \eqref{eq:8-4},  respectively. 
Let $u\in H^6(B_1)$ be a weak solution to \eqref{eq:1-2}. 

Then there exists an absolute constant $R_1\in(0,1]$ such that for every $r\leq s\leq\frac{R_1}{2^8}$, we have

\begin{equation}\label{SUCP-1}
    \int_{B_{s}}u^2\leq
    CN^{\overline{k}}\left(\frac{s}{r}\right)^{\log_2\left(CN^{\overline{k}}\right)} \int_{B_r} u^2,
\end{equation}
where $N$ is given by 
\begin{equation*}
N= \frac{\int_{B_{R_1}} u^2}{\int_{B_{{R_1}/{2^7}}}u^2} .
\end{equation*}

In addition, if $2r\leq s\leq\frac{R_1}{2^8}$ then we have

\begin{equation}\label{Tre-sfere-U}
    \int_{B_{s}}u^2\leq
    \left(C\int_{B_{R_1}}u^2\right)^{1-\widetilde{\theta}(s,r)}\left(\int_{B_r}u^2\right)^{\widetilde{\theta}(s,r)},
\end{equation}
where
\begin{equation*}
   \widetilde{\theta}(s,r)=\frac{1}{1+2\overline{k}\log_2\frac{s}{r}}
\end{equation*}
(with $\bar{k}=8$) and $C$ only depends on $M_2, \alpha_0, \gamma_0, t, l$. 
\end{prop}
\begin{proof}
See Corollary 4.10 in \cite{MRSV2022}. 
\end{proof}

\begin{lem}[Caccioppoli-type inequality]
    \label{lem:intermezzo-cube}
    Let us assume that the hypothesis of Proposition \ref{prop:SUCP-1} are satisfied. 
Then, for every $r$, $0<r<1$, we have
        \begin{equation}
    \label{eq:12a.2-cube}
    \|D^hu\|_{L^2(B_{\frac{r}{2}})}\leq \frac{C}{r^h}\|u\|_{L^2(B_r)}, \quad \forall
    h=1, ..., 6,
\end{equation}
where $C$ is a constant only depending on $M_2,\alpha_0, \gamma_0, t,l$ only.
\end{lem}
\begin{proof}
For the proof we refer to \cite[Lemma 4.7]{MRSV2022}.
\end{proof}

Let us now recall a Poincar\'{e}-type inequality. Let $R,r$ positive
numbers such that $r\leq R$. For a given function $u\in H^2(B_R)$
denote

\begin{equation}
    \label{6-1-Indiana}
    (u)_r=\frac{1}{\left|B_r\right|}\int_{B_r}u, \quad\quad (D u)_r=\frac{1}{\left|B_r\right|}\int_{B_r}D u
\end{equation}
and
\begin{equation}
    \label{6-2-Indiana}
    \widetilde{u}_r=u(x)-(u)_r-(D u)_r\cdot x
\end{equation}

\begin{prop}[Poincar\'{e} inequality]\label{Poincare} There
exits a positive absolute constant $C$ such that

\begin{equation}
    \label{6-3-Indiana}
    \int_{B_R}\left|\widetilde{u}_r\right|^2+R^2
    \int_{B_R}\left|D\widetilde{u}_r\right|^2\leq
    C\frac{R^6}{r^2}\int_{B_R}\left|D^2\widetilde{u}_r\right|^2,
\end{equation}
for every $u\in H^2(B_R)$ and for every $r\in (0,R]$.
\end{prop}
\begin{proof}
See \cite{AMR2008} and \cite[Proposition 6.1]{MRV2007}.
\end{proof}

\begin{prop} [{\bf Doubling inequality and three sphere inequality for the Hessian}]
    \label{theo:40.teo-sf}
    Let us assume that the hypothesis of Proposition \ref{prop:SUCP-1} are satisfied. 
    Then there exists $C>1$, only depending on $M_2,\alpha_0,\gamma_0, t,l$ only, such that, for every
$0<r<\frac{R_1}{2^{11}}$ we have
\begin{equation}\label{eq:10.6.1102-cube-sf}
    \int_{B_{2r}}\left|{D}^2u\right|^2\leq C \overline{N}^{3\overline{k}}\int_{B_{r}}\left|{D}^2u\right|^2,
\end{equation}

where
\begin{equation}
    \label{eq:10.6.1108-cube-sf}
    \overline{N}=\frac{\left\Vert
{D}^2u\right\Vert^2_{L^2\left(B_{R_1}\right)}}{\left\Vert
{D}^2u\right\Vert^2_{L^2\left(B_{R_1/2^9}\right)}}.
    \end{equation}
In addition, if $2r\leq s\leq\frac{R_1}{2^{11}}$ then we have

\begin{equation}\label{Tre-sfere-hessU}
    \int_{B_{s}}\left\vert {D}^2 u\right\vert^2\leq
    \left(C\int_{B_{R_1/2}}\left\vert {D}^2 u\right\vert^2\right)^{1-\theta(s,r)}\left(\int_{B_r}\left\vert {D}^2 u\right\vert^2\right)^{\theta(s,r)},
\end{equation}
where
\begin{equation}\label{theta}
   \theta(s,r)=\frac{1}{1+6\overline{k}\log_2\frac{s}{r}}.
\end{equation}
(with $\overline{k}=8$).

\end{prop}
\begin{proof}
 Let $$0<4r<\frac{R_2}{2^8},$$
with $R_2=\frac{R_1}{2}$.
We define $v=\widetilde{u}_{R_2}$ and we observe that since $|{D}^2u|=|{D}^2 v|$ we may as well prove \eqref{eq:10.6.1102-cube-sf}  and \eqref{Tre-sfere-hessU} for $v$ instead. 

Let us note that $v$ is still a solution to \eqref{eq:1-2}.

Hence by Lemma \ref{lem:intermezzo-cube} we
have that for every $r\in (0,R_1]$ the following holds 

\begin{equation}\label{doubl-sf}
  r^4\int_{B_{2r}}\left|{D}^2 v\right|^2=  r^4\int_{B_{2r}}\left|{D}^2\widetilde{v}_r\right|^2\leq C \int_{B_{4r}}\left|\widetilde{v}_r\right|^2 \ .
\end{equation}
By \eqref{6-3-Indiana} we have

\begin{equation}
    \label{doubl-sf-1a}
    \int_{B_r}\left|\widetilde{v}_r\right|^2\leq
    Cr^4\int_{B_r}\left|{D}^2\widetilde{v}_r\right|^2=Cr^4\int_{B_r}\left|{D}^2 v \right|^2
\end{equation}
Now, denote by

\begin{equation}
    \label{doubl-sf-1}
    \widetilde{N}_r=\frac{\int_{B_{R_2}}\left|\widetilde{v}_r\right|^2}{\int_{B_{R_2/2^7}}\left|\widetilde{v}_r\right|^2}.
    \end{equation}
By \eqref{SUCP-1}, \eqref{doubl-sf},  \eqref{doubl-sf-1a} and \eqref{doubl-sf-1}  we have

\begin{equation*}
r^4\int_{B_{2r}}\left|{D}^2v\right|^2\leq C
\int_{B_{4r}}\left|\widetilde{v}_r\right|^2\leq
    C\widetilde{N}_r^{3\overline{k}}\int_{B_{r}}\left|\widetilde{v}_r\right|^2\leq Cr^4\widetilde{N}_r^{3\overline{k}}\int_{B_r}\left|{D}^2v\right|^2.
\end{equation*}
Hence, for every $r$ that satisfies $0<4r<\frac{R_2}{2^8}$, we have
\begin{equation}
    \label{doubl-sf-2}
    \int_{B_{2r}}\left|{D}^2v\right|^2\leq C
    \widetilde{N}_r^{3\overline{k}}\int_{B_r}\left|{D}^2v\right|^2.
    \end{equation}
Now we estimate $\widetilde{N}_r$ from above. By Lemma
\ref{lem:intermezzo-cube} we have

\begin{equation}
    \label{doubl-sf-3}
    \int_{B_{R_2/2^7}}\left|\widetilde{v}_r\right|^2\geq
    \frac{1}{C}\left(\frac{R_2}{2^7}\right)^4\int_{B_{R_2/2^8}}\left|{D}^2\widetilde{v}_r\right|^2=
    \frac{1}{C}\left(\frac{R_2}{2^7}\right)^4\int_{B_{R_2/2^8}}\left|{D}^2v\right|^2.
    \end{equation}
Moreover, by triangle inequality and the Sobolev embedding Theorem, \cite[Ch.
7]{G-Tr}, we get

\begin{equation}\label{doubl-sf-4}
    \begin{aligned}
    \left\Vert\widetilde{v}_r\right\Vert_{L^2\left(B_{R_2}\right)}&\leq \left\Vert
    v\right\Vert_{L^2\left(B_{R_2}\right)}+
    CR_2\left\vert(v)_r\right\vert+CR^2_2\left\vert(D
    v)_r\right\vert\leq\\&\leq CR_2\left\Vert
    v\right\Vert_{L^{\infty}\left(B_{R_2}\right)}+CR^2_2\left\Vert
    D v\right\Vert_{L^{\infty}\left(B_{R_2}\right)}\leq\\&\leq
    C\left\Vert
     v\right\Vert_{H^{3}\left(B_{R_2}\right)}.
    \end{aligned}
    \end{equation}
     
 Therefore,   by \eqref{doubl-sf-4}   and \eqref{6-3-Indiana} we have 
 \begin{equation}  
  \int_{B_{R_2}}\left|\widetilde{v}_r\right|^2\le C  R^4_2\left\Vert
     {D}^2 v\right\Vert^2_{L^{2}\left(B_{R_2}\right)}  + C R^6_2\left\Vert
     {D}^3 v\right\Vert^2_{L^{2}\left(B_{R_2}\right)},
     \end{equation}
     
and, by  \eqref{eq:12a.2-cube}, we have that 

 \begin{equation}\label{stimaalto1} 
  \int_{B_{R_2}}\left|\widetilde{v}_r\right|^2\le C  R^4_2\left\Vert
     {D}^2 v\right\Vert^2_{L^{2}\left(B_{R_2}\right)}  + C \left\Vert
      v\right\Vert^2_{L^{2}\left(B_{2R_2}\right)}.
     \end{equation}

Using again \eqref{6-3-Indiana} we have that

\begin{equation} \label{stimaalto2} 
  \int_{B_{R_2}}\left|\widetilde{v}_r\right|^2\le C  R^4_2\left\Vert
     {D}^2 v\right\Vert^2_{L^{2}\left(B_{2R_2}\right)}  \ . 
          \end{equation}

Hence combining  \eqref{stimaalto2} and \eqref{doubl-sf-3} we have that 
\begin{equation}\label{doubl-sf-7}
    \widetilde{N}_r\leq  \frac{C\left\Vert
{D}^2v\right\Vert^2_{L^2\left(B_{2R_2}\right)}}{\left\Vert
{D}^2v\right\Vert^2_{L^2\left(B_{R_2/2^8}\right)}} \ .
    \end{equation}

Now, recalling that $R_2=\frac{R_1}{2}$ and that $|{D}^2v|=|{D}^2u|$, by \eqref{doubl-sf-2}  we get \eqref{eq:10.6.1102-cube-sf}.

Finally, by using the same argument used in Corollary 4.10 in \cite{MRSV2022}, by \eqref{eq:10.6.1102-cube-sf} we obtain
\eqref{Tre-sfere-hessU} easily.

\end{proof}

\section{Proof of Theorem \ref{theo:14-1}}
\label{sec:Proof-upper-bound-di-D-per-D-grossa}

\begin{prop}[{\bf Lipschitz propagation of smallness}]\label{LPS}
Let $\Omega$ be a bounded domain in ${\mathbb{R}}^2$, such that $\partial \Omega$ is of class $C^{3,1}$, with constants $r_0, M_0$ and satisfying \eqref{eq:1-1}.
 Let the tensor $\mathbb{P},\mathbb{P}^h \in C^{1,1}(\overline{\Omega}, \mathcal{L}(\widehat{\M}^2, \widehat{\M}^2)), \mathbb{Q} \in C^{2,1}(\overline{\Omega}, \mathcal{L}(\widehat{\M}^3, \widehat{\M}^3))$, be given by \eqref{eq:6-1}, \eqref{eq:6-2} and \eqref{eq:7-0} respectively and satisfying the strong convexity conditions \eqref{eq:8-1}. Let $u_0 \in H^3(\Omega)$ be the unique solution to the problem \eqref{eq:1-2} - \eqref{eq:1-5} normalized by \eqref{eq:4-3}, with Neumann data $\widehat{V} \in H^{ - 3/2  }(\partial \Omega), \widehat{M}_n \in H^{ - 1/2  }(\partial \Omega), \widehat{M}_n^h \in H^{ 1/2  }(\partial \Omega)$ satisfying the compatibility condition \eqref{eq:4-2}. There exists $\chi>1$ only depending on $\alpha_0, \gamma_0, M_2$ and $\frac{t}{r_0}$ such that for every $s>0$ and for every $x\in \Omega_{\chi s {r_0}}$ we have that 
 \begin{equation}\label{eq:18-1}
 \int_{B_{s r_0(x)}}|D^2 u_0|^2 \ge C_s  \int_{\Omega}|D^2 u_0|^2 \ , 
 \end{equation}
wuth $C_s>0$ only depending on $M_0, M_1, \frac{t}{r_0}, \frac{l}{r_0}, \alpha_0, \gamma_0, M_2, s$ and on the ratio $F$ given in \eqref{eq:14-4}.
 \end{prop}

\begin{lem}\label{ex71}
Let $\Omega$ be a bounded domain in $\mathbb{R}^2$ such that $\partial \Omega$ is of class $C^{3,1}$ with constants $r_0, M_0$ satisfying \eqref{eq:1-1}. Let the tensors 
$\mathbb{P},\mathbb{P}^h \in L^{\infty}({\Omega}, \mathcal{L}(\widehat{\M}^2, \widehat{\M}^2)), \mathbb{Q} \in L^{\infty}({\Omega}, \mathcal{L}(\widehat{\M}^3, \widehat{\M}^3))$ satisfy the symmetry conditions \eqref{eq:3-1}, \eqref{eq:3-2} and \eqref{eq:3-6} and the strong convexity assumptions \eqref{eq:3-3} and \eqref{eq:3-7}. 
Let $u_0\in H^3(\Omega)$ be the unique weal solution to problem \eqref{eq:1-2} - \eqref{eq:1-5}, satisfying the normalization condition \eqref{eq:4-3} with the boundary data satisfying \eqref{eq:4-1} and \eqref{eq:4-2}. We have 
\begin{equation}\label{eq:18-2}
\|\widehat{V}\|_{H^{ - 5/2  }(\partial \Omega)}+ r_0^{-1} \|\widehat{M}_n\|_{H^{ - 3/2  }(\partial \Omega)} +  r_0^{-2}\|\widehat{M}_n^h\|_{H^{ -1/2  }(\partial \Omega) }\le C \|u_0\|_{H^3(\Omega)} \ , 
\end{equation}
where $C>0$ depends on $M_0, M_1, \|\mathbb{P} \|_{L^{\infty}(\Omega)}, \|\mathbb{P}^h \|_{L^{\infty}(\Omega)}, \|\mathbb{Q} \|_{L^{\infty}(\Omega)}$.
 \end{lem}
\begin{proof}
Let us estimate the first term on the left hand side of \eqref{eq:18-2}. Similar arguments allow to estimate the other two terms. 
Given $g\in H^{ 5/2  }(\partial \Omega)$, by extension results there exists $w\in H^3(\Omega)$ be such that $w=g, w_{,n}=0, w_{,nn}=0$ on $\partial \Omega$, and 
\begin{equation}\label{eq:18-3}
\|w\|_{H^3(\Omega)}\le C \|g\|_{H^{ 5/2  }(\partial \Omega)}
\end{equation}
where $C>0$ depends on $M_0, M_1$ (see for example \cite{Necas67}). 

We have 
\begin{eqnarray}\label{eq:18-4}
&&\int_{\partial \Omega} \widehat{V} g = \int_{\partial \Omega} \widehat{V}w + \widehat{M}_nw_{,n} + \widehat{M}^h_nw_{,nn}=
-\int_{\Omega} (\mathbb{P}+\mathbb{P}^h)D ^2 u_0 \cdot D ^2 w + \mathbb{Q} D^3 u_0 \cdot D^3 w \le \nonumber \\
&&C r_0^5 \left( \|D^2 u_0\|_{L^2(\Omega)} \cdot \|D^2 w\|_{L^2(\Omega)} + r_0^2\|D^3 u_0 \|_{L^2(\Omega)}\cdot \| D^3 w\|_{L^2(\Omega)}\right)\le \nonumber \\
&&C r_0\|u_0\|_{H^3(\Omega)}\cdot \|w\|_{H^3(\Omega)}\le C r_0 \|u_0\|_{H^3(\Omega)}\cdot \|g\|_{H^{ 5/2  }(\partial \Omega)}.
\end{eqnarray}
Therefore, by \eqref{eq:18-4},
\begin{equation}
\|\widehat{V}\|_{H^{ - 5/2  }(\partial \Omega)} = \sup_{\|g\|_{H^{ 5/2  }(\partial \Omega)}=1} \frac{1}{r_0} \int_{\partial \Omega}\widehat{V}g \le C \|u_0\|_{H^3(\Omega)}.
\end{equation}
\end{proof}

\begin{proof} [Proof of Proposition \ref{LPS}]
Let us assume for this proof $r_0=1$.  
By following the lines of the proof of Proposition 5.2 in \cite{MRV2007}, a process of iteration of the three spheres inequality \eqref{Tre-sfere-hessU} leads to
\begin{equation}\label{eq:18-5}
\frac{\|D^2 u_0\|_{L^2(\Omega_{(\chi+1)s)})}}{\|D^2 u_0\|_{L^2(\Omega)}}\le \frac{C}{s}\left( \frac{\|D^2 u_0\|_{L^2(B_s(x))}}{\|D^2 u_0\|_{L^2(\Omega)}} \right)^{{\theta_0}^{L-1}}
\end{equation}
for every $s\leq \frac{s_0}{\chi}$.
Here $\theta_0\in (0,1)$, $C>0$ and $\chi>1$ only depend on $\alpha_0, \gamma_0, M_0, M_1, M_2, t, l$; $s_0$ only depends on $M_0$ and is such that $\Omega_{\chi s}$ is connected for $s\leq \frac{s_0}{\chi}$ (see for instance Proposition $5.5$ in \cite{A-R-R-V}); $0<L<\frac{M_1}{\pi s^2}$.

Let us rewrite the left hand side of \eqref{eq:18-5} as follows 

\begin{equation}\label{eq:18-6}
\frac{\|D^2 u_0\|^2_{L^2(\Omega_{(\chi+1)s)})}}{\|D^2 u_0\|^2_{L^2(\Omega)}}= 1 - \frac{\int_{\Omega\setminus \Omega_{(\chi+1)s}}|D^2 u_0|^2}{\int_{\Omega}|D^2 u_0|^2} \ .
\end{equation}
By H\"{o}lder and Sobolev inequalities we have that 
\begin{eqnarray}\label{eq:18-7}
&&\|D^2 u_0 \|^2_{L^2(\Omega \setminus \Omega_{(\chi +1)s})}\le |\Omega \setminus |\Omega_{(\chi +1)s}|^{\frac{1}{2}}\|D^2 u_0 \|^2_{L^4(\Omega \setminus \Omega_{(\chi +1)s} )}\le 
 \nonumber \\
 &&C s^{\frac{1}{2}}\|D^2 u_0 \|_{H^{ 1/2  }(\Omega)}^2 \le C  s^{\frac{1}{2}}\|u_0\|^2_{H^3(\Omega)} \ ,
\end{eqnarray}
where $C>0$ only depends on $M_0, M_1$, $\alpha_0$, $\gamma_0$, $M_2$, $t$, $l$.

Let us notice that in the last step we have used the bound 

\begin{equation}\label{eq:18-8}
| {\Omega\setminus \Omega_{(\chi+1)s}}|\le C s \ , 
\end{equation}
where $C>0$ only depends on $M_0$ (see \cite{AR1998} ). 

Let us recall the following interpolation inequality: for any $u\in H^4(\Omega)$, we have 
\begin{equation}\label{eq:18-9}
\|u\|_{H^3(\Omega)}\le C \|u\|^{\frac{1}{2}}_{H^2(\Omega)}\|u\|^{\frac{1}{2}}_{H^4(\Omega)}
\end{equation}
(see \cite[Theorem 7.25]{G-Tr}), where $C>0$ depends on $M_0, M_1$ only.

By \eqref{eq:18-7}, standard Poincar\`e inequality, \eqref{eq:18-9}, \eqref{eq:5bis-1} and Lemma \ref{ex71}, we have 
\begin{equation}\label{eq:18-10}
\frac{\int_{\Omega\setminus \Omega_{(\chi+1)s}} |D^2u_0|^2 }{\int_{\Omega} |D^2 u_0|^2}\le C s^{\frac{1}{2}} \frac{\|u_0\|^2_{H^3(\Omega)}}{\|u_0\|^2_{H^2(\Omega)}}\le
 C s^{\frac{1}{2}} \left( \frac{\|u_0 \|_{H^4(\Omega)}}{\|u_0 \|_{H^3(\Omega)}}\right)^2\le C s^{\frac{1}{2}} F^2\le \frac{1}{2}
\end{equation}
for $s\le \bar{s}$, where $\bar{s}$ only depends on $M_0,M_1,t,l,\alpha_0,\gamma_0,M_2$ and $F$. 
\end{proof}
Finally, the thesis follows from \eqref{eq:18-5}, \eqref{eq:18-6} and \eqref{eq:18-10}.

\begin{proof}[Proof of Theorem \ref{theo:14-1}]
We can cover $D_{h_1r_0}$ with internally non-overlapping closed squares $Q_l$ of side $\epsilon r_0$, where $l=1,\dots, L$ and $\epsilon=\mbox{min} \left\{\frac{2h_1}{\chi + \sqrt{2}},\frac{h_1}{\sqrt{2}}\right\}$, where $\chi>1$ has been defined in Proposition \ref{LPS}. By construction, all the squares are contained in $D$ and $|D_{h_1r_0}|\le L\epsilon^2r_0$ .
 Let $\bar{l}$ be such that 
$\int_{Q_{\bar{l}}}|D^2 u_0|^2=\min_{l=1,\dots,L} \int_{Q_{l}}|D^2 u_0|^2$. By the fatness assumption \eqref{eq:14-1}, we have 
\begin{equation}\label{eq:18-11}
\int_D |D^2 u_0|^2 \ge \frac{|D|}{2r_0^2\epsilon^2}\int_{Q_{\bar{l}}}|D^2 u_0|^2 \ .
\end{equation}
Let $\overline{x}$ be the center of the square $Q_{\overline{l}}$.
By applying the Lipschitz propagation of smallness estimate \eqref{eq:18-1} with $x=\bar{x}$ and $s=\frac{\epsilon}{2}$, we have 
\begin{equation}\label{eq:18-12}
\int_D |D^2 u_0|^2 \ge \frac{C|D|}{r_0^2}\int_{\Omega}|D^2 u_0|^2 \,
\end{equation}
with $C$ only depending on $h_1$, $\alpha_0$, $\gamma_0$, $M_0$, $M_1$, $M_2$, $\frac{t}{r_0}$, $\frac{l}{r_0}$, $F$.
By \eqref{eq:18-12}, by applying Poincar\`e inequality, interpolation inequality \eqref{eq:18-9}, the regularity estimate \eqref{eq:5bis-1}, Lemma \ref{ex71}, \eqref{eq:16-1}-\eqref{eq:16-2} and the weak formulation of the problem \eqref{eq:1-2}-\eqref{eq:1-5}, we have 
\begin{eqnarray}\label{eq:18-13}
\int_{D} |D^2 u_0|^2 \ge C \frac{|D|}{r_0^4} \|u_0\|^2_{H^2(\Omega)}\ge C\frac{|D|}{r_0^4} \frac{ \|u_0\|^2_{H^3(\Omega)}}{ \|u_0\|^2_{H^4(\Omega)}} \|u_0\|^2_{H^3(\Omega)}\nonumber\\
\ge C |D| F^{-2} (\|D^2 u_0\|^2_{L^2(\Omega)} + r_0^2 \|D^3 u_0  \|_{L^2(\Omega)}^2)\ge C |D| F^{-2} r_0^{-5}W_0 \ ,
\end{eqnarray}
where $C>0$ depends on $M_0, M_1, \frac{t}{r_0}, \frac{l}{r_0}, \alpha_0, \gamma_0, M_2, h_1$. 
 Estimates \eqref{eq:14-2} and \eqref{eq:14-3} follow {from} \eqref{eq:18-13} and {from} the left hand side of \eqref{eq:17-1} and \eqref{eq:17-2} respectively.
\end{proof}

\section{Proof of Theorem \ref{theo:15-1}}
\label{sec:Proof-upper-bound-di-D-per-D-generica}

\begin{prop}[Doubling inequality for the Hessian in terms of the boundary data]\label{DoublingData}
Under the hypothesis of Theorem \ref{theo:15-1}, let $u_0\in H^{3}(\Omega)$ be the unique solution to \eqref{eq:1-2}-\eqref{eq:1-5} satisfying \eqref{eq:4-3}, with $\widehat{V}, \widehat{M}_n, \widehat{M}^h_n$ satisfying \eqref{eq:4-1} and \eqref{eq:4-2}. There exists a constant $\theta, 0<\theta<1$, only depending on $\alpha_0, \gamma_0, M_2, \frac{t}{r_0}, \frac{l}{r_0}$, such that for every $\bar{r}>0$ and for every $x_0\in \Omega_{\bar{r}r_0}$, we have 
\begin{equation}\label{eq:19-1}
\int_{B_{2r}(x_0)} |D^2 u_0|^2 \le K \int_{B_{r}(x_0)} |D^2 u_0|^2
\end{equation}
for every $r, 0<r<\frac{\theta}{2}\bar{r}r_0$, where $K>0$ only depends on $\alpha_0, \gamma_0, M_2, M_0, M_1, \bar{r}, \frac{t}{r_0},\frac{l}{r_0}$ and the ratio $F$ given by \eqref{eq:14-4}.

\end{prop}
\begin{proof}
By applying a scaling argument to \eqref{eq:10.6.1102-cube-sf} and \eqref{eq:10.6.1108-cube-sf}, there exists an absolute constant $\theta, 0<\theta<1$ such that for every $\bar{r}>0$ and for every $x_0\in \Omega_{\bar{r}r_0}$ we have 
\begin{equation}\label{eq:19-1-2}
\int_{B_{2r}(x_0)}|D^2 u_0|^2 \le K \int_{B_{r}(x_0)}|D^2 u_0 |^2
\end{equation}
for every $r, 0<r<\frac{\theta}{2}\bar{r}r_0$, where $K>0$ only depends on $\alpha_0, \gamma_0, M_2, \frac{t}{r_0}, \frac{l}{r_0}$ and $\bar{r}$ and the increasingly on the ratio 
\begin{equation}\label{eq:19-2}
N=\frac{\int_{B_{\bar{r}r_0}(x_0)}|D^2u_0|^2}{\int_{B_{\frac{\bar{r}r_0}{2^9}}(x_0)}|D^2u_0|^2} .
\end{equation}
By applying \eqref{eq:18-1} to bound from below the denominator, we trivially obtain the desired bound. 

\end{proof}

\begin{prop}
\label{prop:Ap} {\rm ($A_p$ property)} Let the assumptions of
Proposition \ref{DoublingData} be satisfied. For every
$\overline{r}>0$ there exist $B>0$ and $p>1$ such that for every
$x_0\in\Omega_{\overline{r}r_0}$ we have
\begin{multline}
    \label{eq:Pesi-Ap}
    \left
    (\frac{1}
    {|B_{r}(x_0)|}\int_{B_{r}(x_0)}|D^2 u_0|^2
    \right)
    \left(\frac{1} {|B_{r}(x_0)|}\int_{B_{r}(x_0)}
    |D^2 u_0|^{-2/(p-1)}
    \right)^{p-1}
    \leqslant B, \\
    \qquad\qquad\qquad\qquad \textrm{for every}~ r,\ 0<r\leqslant
    \frac{\vartheta}{2}\overline{r}r_0,
\end{multline}
where $\vartheta$ is as in Proposition \ref{DoublingData} and
where $B$, $p$ only depend on $\alpha_0, \gamma_0, M_2, M_0, M_1, \bar{r}, \frac{t}{r_0},\frac{l}{r_0}$ and the ratio $F$ given by \eqref{eq:14-4}.
\end{prop}
\begin{proof}
In view of the results in \cite{Coi-Fe}, it is sufficient to prove a
reverse H\"older's inequality for $|D^2 u_0|^2$. Let us
introduce
\begin{equation}
    \label{eq:def-v0}
    v_0=u_0 +ax_1+bx_2+c,
\end{equation}
such that
\begin{equation}
    \label{eq:norm-v0}
    \int_{B_{2r}(x_0)} v_0 dx=0, \quad \int_{B_{2r}(x_0)} v_0,_\alpha
    dx=0,\quad \alpha=1,2.
\end{equation}
By interior regularity estimates (see, for instance, \cite[Theorem
8.3]{MRV2007}, by Poincar\'{e} inequality (see, for instance,
\cite[Proposition 3.3]{MRV2007}) and by Proposition
\ref{DoublingData} we have
\begin{multline}
\label{eq:reverse} \|D^2 u_0\|_{L^\infty(B_r(x_0))}=
\|D^2 v_0\|_{L^\infty(B_r(x_0))}\leq
\frac{C} {r^2}\|v_0\|_{H^2(B_{2r}(x_0))}\leq \\
\leq C\|D^2 v_0\|_{L^2(B_{2r}(x_0))} =C\|D^2
u_0\|_{L^2(B_{2r}(x_0))}\leq C \|D^2 u_0\|_{L^2(B_{r}(x_0))},
\end{multline}
where $C$ only depends on $\alpha_0, \gamma_0, M_2, M_0, M_1, \bar{r}, \frac{t}{r_0},\frac{l}{r_0}$ and the ratio $F$ given by \eqref{eq:14-4}.
\end{proof}

\begin{proof}[Proof of Theorem \ref{theo:15-1}]

Let us cover $D$ with internally non overlapping closed cubes
$Q_{j}$, $j=1,...,J$, with side $\epsilon=\frac {{\theta} d_{0}}{4\sqrt 2}r_0$, where
$\theta<1$ has been introduced in Proposition
\ref{DoublingData}. Let $p>1$  be the exponent introduced in Proposition \ref{prop:Ap}.  By H\"{o}lder's
inequality we have
\begin{equation}
  \label{eq:main-2}
  |D|
  \leq
  {\left (
  \int_{\bigcup_{j=1}^{J} Q_{j}} {| {D}^2 u_{0}|}^{- \frac {2}{p-1}}
  \right )}^{\frac {p-1}{p}}
  {\left (
  \int_{D} {| {D}^2 u_{0}|}^{2}
  \right )}^{\frac {1} {p}}.
\end{equation}
 By applying
Proposition \ref{prop:Ap}, with $\bar r=\frac{d_0}{2}$ to
the balls $B_{j}$ circumscribing each $Q_{j}$, $j=1,...,J$, we
have

\begin{multline}
   \label{eq:main-3}
  {\left (
  \int_{\bigcup_{j=1}^{J} Q_{j}} {| {D}^2 u_{0}|}^{- \frac {2}{p-1}}
  \right )}^{\frac {p-1}{p}}
  \leq
  {\left( \frac{\pi}{2}{\epsilon}^{2} \sum_{j=1}^{J}
  \frac {1} {|B_{j}|}
  \int_{B_{j}} {| {D}^2 u_{0}|}^{- \frac {2}{p-1}}
  \right )}^{\frac {p-1}{p}}
  \leq
  \\
  \leq
  {\left ( \frac{\pi}{2}{\epsilon}^{2} \sum_{j=1}^{J}
  {\left (
  \frac
  {B}
  {\frac {1}
  {|B_{j}|} \int_{B_{j}} {| {D}^2 u_{0}|}^{2}}
  \right )}^{\frac {1}{p-1}}
  \right )}^{\frac {p-1}{p}}
  \leq
  \frac
  {\frac{\pi}{2}\left (J{\epsilon}^{2} \right )^{\frac {p-1}{p}} {B}^{\frac {1}{p}}\epsilon^{2/p} }
  {
  \min_{j}
  \left(\int_{B_{j}} {| {D}^2 u_{0}|}^{2}
  \right )^{\frac {1} {p}}
                     },
\end{multline}
where the constants $p$ and $B$ only depend on $\alpha_0$,
$\gamma_0$, $M_2$, $M_1$, $M_0$, $\frac{t}{r_0}, \frac{l}{r_0}$, $d_0$ and  the ratio $F$ given by \eqref{eq:14-4}. By
\eqref{eq:1-1} we have 
\begin{equation}
  \label{eq:main-4}
   J{\epsilon}^{2}=\sum_{j=1}^{J} {|Q_{j}|} \leq |\Omega|\leq
   M_1r_0^2.
\end{equation}
Consequently from \eqref{eq:main-2}--\eqref{eq:main-4} and recalling the definition of $\epsilon$, we have
\begin{equation}
  \label{eq:main-5}
  |D|
  \leq
  Cr_0^2
  {\left (
  \frac
  { \int_{D} {| {D}^2 u_{0}|}^{2} }
  {\int_{B_{\overline{j}}} {| {D}^2 u_{0}|}^{2} }
  \right )} ^{\frac {1} {p}},
\end{equation}
with $\overline{j}$ such that $\int_{B_{\overline{j}}} {| {D}^2 u_{0}|}^{2}=\min_j \int_{B_{j}} {| {D}^2 u_{0}|}^{2}$.

By Proposition \ref{LPS}, by standard Poincar\`e inequality, by the interpolation inequality \eqref{eq:18-9}, by \eqref{eq:18-2}, \eqref{eq:5bis-1}, \eqref{eq:9-5} and \eqref{eq:11-4}, we have   
\begin{multline}
  \label{eq:main-6}
\int_{B_{\overline{j}}} {| {D}^2 u_{0}|}^{2} \geq C \int_{\Omega} {| {D}^2 u_{0}|}^{2}
\geq \frac{C}{r_0^2}\|u_0\|^2_{H^2(\Omega)}\geq\\
\geq\frac{C}{r_0^2}\left(\frac{\|u_0\|_{H^3(\Omega)}}{\|u_0\|_{H^4(\Omega)}}\right)^2\cdot \|u_0\|_{H^3(\Omega)}^2
\geq CF^{-2}\left(\int_\Omega {| {D}^2 u_{0}|}^{2}+r_0^2{| {D}^3 u_{0}|}^{2}\right)
\geq \frac{C}{r_0^3}W_0.
\end{multline}
By \eqref{eq:main-5}, \eqref{eq:main-6},
\begin{equation}
  \label{eq:main-7}
  |D|
  \leq
  Cr_0^2
  {\left (
  \frac
  { r_0^3\int_{D} {| {D}^2 u_{0}|}^{2} }
  {W_0 }
  \right )} ^{\frac {1} {p}},
\end{equation}
with the constant $C>0$ only depending on $\alpha_0$,
$\gamma_0$, $M_2$, $M_1$, $M_0$, $\frac{t}{r_0}, \frac{l}{r_0}$, $d_0$ and  the ratio $F$ given by \eqref{eq:14-4}. 

Finally, from the left hand sides of
\eqref{eq:17-1} and \eqref{eq:17-2} and from
\eqref{eq:main-7} we end up with the upper bounds for $|D|$ in \eqref{eq:15-1} and
\eqref{eq:15-2}.

\end{proof}

\bigskip

\bibliographystyle{plain}

\end{document}